\documentclass[a4paper]{article}
\usepackage[dvips]{graphics,graphicx}
\usepackage{amsfonts,amssymb,amsmath,color,mathrsfs, amstext}
\usepackage{amsbsy, amsopn, amscd, amsxtra, amsthm,authblk}
\usepackage{enumerate,algorithmicx,algorithm}
\usepackage{algpseudocode}
\usepackage{upref}

\usepackage{geometry}
\usepackage[displaymath]{lineno}
\usepackage{float}
\usepackage{yhmath}
\usepackage{booktabs}
\usepackage{subcaption}
\usepackage{multirow}
\usepackage{makecell}
\usepackage[colorlinks,
            linkcolor=blue,
            anchorcolor=blue,
            citecolor=blue
            ]{hyperref}

\newcommand{\tcm}[1]{\textcolor{black}{#1}}

\newtheorem{theorem}{Theorem}[section]
\newtheorem*{theorem*}{Theorem}
\newtheorem{lemma}{Lemma}[section]

\newtheorem{remark}{Remark}[section]

\newtheorem{assumption}{Assumption}[section]

\numberwithin{equation}{section}

\begin{document}

\title{Optimal long-time decay rate of solutions  of complete monotonicity-preserving schemes for nonlinear time-fractional evolutionary equations
\thanks{The research of Dongling Wang is supported in part by National Natural Science Foundation of China under grants 11871057 and 91630205.
The research of Martin Stynes is supported in part by the National Natural Science Foundation of China under grants 12171025 and NSAF-U1930402.}}

\author{Dongling Wang\thanks{School of Mathematics and Computational Science, Xiangtan University, Xiangtan, Hunan 411105, China (\texttt{wdymath@xtu.edu.cn})} 
\ and
Martin Stynes\thanks{Applied and Computational Mathematics Division, Beijing Computational Science Research Center,
Beijing 100193, China (\texttt{m.stynes@csrc.ac.cn}). Corresponding author.}
}

%\date{}
\maketitle

\begin{abstract}
The solution of the nonlinear initial-value problem $\mathcal{D}_{t}^{\alpha}y(t)=-\lambda y(t)^{\gamma}$ for $t>0$ with $y(0)>0$, where  
$\mathcal{D}_{t}^{\alpha}$ is a Caputo derivative of order $\alpha\in (0,1)$ and $\lambda, \gamma$ are positive parameters, is known to exhibit $O(t^{-\alpha/\gamma})$ decay as $t\to\infty$. No corresponding result for any discretisation of this problem has previously been proved. In the present paper it is shown that for the class of complete monotonicity-preserving ($\mathcal{CM}$-preserving) schemes (which includes the L1 and Gr\"unwald-Letnikov schemes) on uniform meshes $\{t_n:=nh\}_{n=0}^\infty$, the discrete solution also has $O(t_{n}^{-\alpha/\gamma})$  decay as $t_{n}\to\infty$. This result is then extended to $\mathcal{CM}$-preserving discretisations of certain time-fractional nonlinear subdiffusion problems such as the time-fractional porous media and $p$-Laplace equations.  For the L1 scheme, the $O(t_{n}^{-\alpha/\gamma})$ decay result is shown to remain valid on a very general class of nonuniform meshes.
 Our analysis uses a discrete comparison principle with discrete subsolutions and supersolutions that are carefully constructed to give tight bounds on the discrete solution. Numerical experiments are provided to confirm our theoretical analysis.
\end{abstract}



\section{Introduction}\label{sec:Introd}

Fractional-order differential equations offer modeling properties that are superior to classical integer-order differential equations in many physical processes with non-local effects or genetic memory characteristics. In particular,  fractional differential operators arise naturally in anomalous diffusion processes from a probability analysis based on a random walk model \cite{gal2020fractional, jin2021fractional}.The solutions of these fractional differential equations are of particular interest when time $t\to t_0$ (the initial time) and when $t\to +\infty$, as in these two regimes they behave very differently from the solutions of classical (integer-order) problems.

As $t\to t_0$, typical solutions of these problems exhibit weak singularities. Obviously this anomaly will affect the behaviour of numerical methods, and this phenomenon has been intensively studied; see for example the survey articles \cite{jin2019numerical,StyL1survey}.

The long-time behaviour of solutions as $t\to\infty$ has received much less attention and this regime is the focus of our paper. We shall consider both fractional ordinary differential equations (F-ODEs) and time-fractional initial-boundary value problems (F-PDEs). The long-time behaviour for a linear F-PDE was investigated in~\cite{GORS_JSC18}, but in the present paper we shall consider nonlinear problems whose analysis is much more difficult, both theoretically and numerically \cite{gal2020fractional, wang2020high}.

\subsection{Linear F-ODEs}
 As time $t\to\infty$,  solutions of time-fractional differential equations usually exhibit asymptotic behaviour that is completely different from the classical integer-derivative case.  While solutions of linear classical ODEs generally have exponential decay rates at equilibrium points,  the solutions of F-ODEs have algebraic decay rates, leading to so-called Mittag-Leffler stability \cite{cuesta2007asymptotic, vergara2015optimal}. 
 
It is of course desirable that as $t\to\infty$, the computed numerical solutions of F-ODEs display the same algebraic decay rates as the exact solutions of the F-ODEs.  For certain linear problems, this property was shown rigorously for certain schemes in \cite{cuesta2007asymptotic,wang2021mittag}. 
The main technical tools used in these papers are the discrete Laplace transform (i.e., generating functions) and discrete Tauberian-type results (i.e., singular analysis of generating functions) --- but these techniques are unsuitable for nonlinear models.

\subsection{Nonlinear F-ODEs}
Let $\mathcal{D}_t^{\alpha}y(t):=\frac{1}{\Gamma(1-\alpha)}\int_{0}^{t} (t-s)^{-\alpha}y'(s)\,ds$ denote the Caputo fractional derivative of order $\alpha\in (0, 1)$. (For general information about fractional calculus see \cite{gal2020fractional, jin2021fractional}.) Our paper will begin by investigating a class of numerical methods for the following nonlinear F-ODE:
\begin{equation}\label{FODE}
\mathcal{D}_{t}^{\alpha}y(t)=-\lambda y(t)^{\gamma} \ \text{ for }t>0, \ \text{with }y(0) = y_0>0,
\end{equation}
where $\lambda$ and $\gamma$ are positive parameters. The asymptotic behaviour of the exact solution of \eqref{FODE} is described in the next result.

\begin{lemma} \cite[Theorem 7.1]{vergara2015optimal} \label{lem:non-decay}
Let $y(t)\in H^{1}_{1, loc}(\mathbb{R}^+)$ be the solution of \eqref{FODE}. 
Then there exist positive constants $C_{1}, C_{2}$, which are independent of $t$, such that 
%\begin{equation}\label{eq:fdc} 
%\begin{split}
% \frac{C_{1}}{1+t^{\alpha/\gamma}}\leq y(t)\leq \frac{C_{2}}{1+t^{\alpha/\gamma}}\ \text{ for } t \geq0.
%\end{split}
%\end{equation}
\end{lemma}

This result agrees with the linear case $\gamma=1$, where the solution can be written in terms of a Mittag-Leffler function as $y(t)=E_{\alpha,1}(-\lambda t^{\alpha})y_0$, which implies that $y(t)=O(t^{-\alpha})$ as $t\to\infty$.

Lemma~\ref{lem:non-decay} shows that the asymptotic decay rates of the solution of \eqref{FODE} with $\alpha\in (0, 1)$ differs significantly from the classical ODE with $\alpha=1$, where one has algebraic decay $y(t)\sim C t^{-1/(\gamma-1)}$ for $\gamma >1$, exponential decay $y(t)= C e^{\lambda t}$ for $\gamma =1$, and extinction in finite time for $\gamma <1$; see \cite{feng2018continuous, vergara2015optimal}. 

\subsection{Time-fractional nonlinear subdiffusion problems}\label{sec:timefracprob}
Lemma~\ref{lem:non-decay} is employed to establish decay rate estimates for exact solutions of several time-fractional nonlinear PDEs in \cite{vergara2015optimal, dipierro2019decay, affili2019decay}.
 Consider the initial-boundary value problem
\begin{equation}\label{eq:FPDEs}
\begin{split}
 \left\{
\begin{aligned}
&\mathcal{D}_{t}^{\alpha}u(x, t)+\mathcal{N}[u](x, t)=0\ \text{ for } x\in\Omega,\, t>0, \\
&u(x, t)=0 \ \text{ for } x\in\mathbb{R}^{d}\setminus \Omega,\, t>0,\\
&u(x, 0)=u_0(x) \ \text{ for } x\in\Omega,
\end{aligned}
\right.
\end{split}
\end{equation}
where $\Omega\subset \mathbb{R}^{d}$ is bounded with smooth boundary and $\mathcal{N}[u]$ is a possibly nonlinear operator. 
%%%
Like~\cite{dipierro2019decay}, we make the following assumptions: the initial data $u_0$ is  nonnegative and does not vanish identically, with  $u_{0}\in L^{\infty}(\mathbb{R}^{d})$ and supp $u_{0}\subset\Omega$, and the solution $u$ is nonnegative with  $u\in L^{q}((0,T), L^{q}(\Omega))$ for every $T>0$ and for every $q\in(1, \infty)$.

Again following \cite{dipierro2019decay}, we make the structural assumption that there exist 
\tcm{$s\in(1, \infty)$,  $\gamma\in(0, \infty)$ 
and $C_s>0$} such that the solution $u$ of \eqref{eq:FPDEs} satisfies
\begin{equation} \label{eq:FPstr}
\tcm{\|u(\cdot, t)\|_{L^{s}(\Omega)}^{s-1+\gamma}\leq C_s \int_{\Omega} u^{s-1}(x, t) \mathcal{N}[u](x, t)\,dx }
	\quad \text{for }t>0, 
\end{equation}
where $\tcm{\|u(\cdot, t)\|_{L^{s}(\Omega)}(t) := \left( \int_{\Omega} |u(x, t)|^{s} dx \right)^{1/s}}$.  
%%%
Roughly speaking, the assumption \eqref{eq:FPstr} is a type of uniform ellipticity of the operator $\mathcal{N}$ with respect to the spatial variable that says the problem has diffusion-like properties.  It is clearly satisfied in the simplest case $\mathcal{N}[u]=-\Delta u$ with $s=2$ and $\gamma=1$. For more examples, including the  time-fractional porous media and $p$-Laplace equations, see \cite{vergara2015optimal, dipierro2019decay, affili2019decay} or Section~\ref{sec:NumExp} below. (Note that the existence of bounded weak solutions of the time-fractional porous media equation and more general nonlinear and degenerate evolutionary integro-differential equations has recently been proved in~\cite{wittbold2021bounded}.)

%%%
Under the above assumptions,  one has the following result.

\begin{lemma} \cite[Theorem 1.1]{dipierro2019decay} \label{lem:non-decay2}
Let $u$ be the solution of \eqref{eq:FPDEs} with the structural condition \eqref{eq:FPstr}. Then 
\begin{align}
\mathcal{D}_{t}^{\alpha}\tcm{\|u\|_{L^{s}(\Omega)}(t)} &\leq - \tcm{\frac{\|u\|_{L^{s}(\Omega)}^{\gamma}(t)}{C_s}} \text{ for all } t>0 \label{eq:FPDEesti}\\
\intertext{and there exists $C_*=C_*(\tcm{C_s},\gamma, \alpha, u_0)>0$ such that} 
 \tcm{\|u\|_{L^{s}(\Omega)}(t)} &\leq \frac{C_{*}}{1+t^{\alpha/\gamma}} \text{ for all } t>0.  \label{eq:FPdecay}
\end{align}
\end{lemma}

%\tcm{
Lemma \ref{lem:non-decay2} exhibits a characteristic attribute of long-time behaviour of solutions of time-fractional PDEs: power-law decay. This phenomenon is fundamentally different from the exponential decay rate of solutions of classical parabolic PDEs. Our paper is mainly concerned with a discrete analogue of Lemma~\ref{lem:non-decay2}.
%}

\subsection{Results and structure of the paper} 
We consider $\mathcal{CM}$-preserving schemes, which were recently introduced in \cite{LiWang2019} as structure-preserving discretisations of Caputo derivatives.  The well-known L1 and Gr\"unwald-Letnikov schemes are particular examples of this class of methods.  $\mathcal{CM}$-preserving schemes  have many desirable properties \cite{LiWang2019}, such as satifying
a discrete fractional comparison principle, enjoying  good numerical stability, and preserving the monotonicity of numerical solutions of nonlinear scalar equations. In the present paper, we shall prove discrete analogues of  Lemmas~\ref{lem:non-decay} and~\ref{lem:non-decay2} for  $\mathcal{CM}$-preserving schemes; that is, we show that the solutions of such schemes are optimal in the sense that they have precisely the same long-time decay rates as the exact solutions of the nonlinear F-ODE~\eqref{FODE} and F-PDE~\eqref{eq:FPDEs}. 

The paper is organised as follows. In Section \ref{sec:CMsche} we define $\mathcal{CM}$-preserving schemes and recall their properties. The discrete analogue of Lemma~\ref{lem:non-decay} on uniform meshes is derived in Section~\ref{sec:discFODE}, using a discrete comparison principle \tcm{to bound the computed solution} by delicately-constructed discrete subsolutions and supersolutions. Section~\ref{sec:L1} considers the particular example of the L1 scheme, and shows that on very general nonuniform meshes one still gets the long-time decay rate of Section~\ref{sec:discFODE}.  The results of Section~\ref{sec:discFODE} are used in Section~\ref{sec:discFPDE} to establish the analogue of Lemma~\ref{lem:non-decay2} on uniform temporal meshes, proving that the $L^s(\Omega)$ norm of computed solutions of discrete-time methods for time-fractional PDEs exhibit the same long-time decay rate as the exact solution, under appropriate structural assumptions. Finally, numerical experiments in Section~\ref{sec:NumExp} demonstrate the sharpness of our theoretical results.\\

\noindent\emph{Notation.} As usual, we use $C$ to represent a generic positive constant, which may take different values at different occasions, but is always independent of $t$ or $n$.
For each $r\in \mathbb{R}$, let $\lceil r\rceil$ denote the smallest integer not less than~$r$.  
\tcm{By $u_n=O(n^{r})$ as $n\to\infty$, we mean that there exists a constant $C>0$ such that $|u_n|\le Cn^r$ for all sufficiently large~$n$.  By $u_{n}\sim v_{n}$ as $n\to\infty$, we mean that  $\lim_{n\to\infty}u_n/v_{n}=C$, where $C>0$ is some constant.
}

\section{$\mathcal{CM}$-preserving schemes}\label{sec:CMsche}

The sequence $\delta_{d}:=(1, 0, 0,...)$ is the convolutional identity: one has $u*\delta_d = \delta_d*u = u$ for any sequence $\omega=\{\omega_n\}_{n=0}^\infty$.
%\footnote{M: If Section 2 is removed, then need to move the definition of discrete convolution to here.}
The \emph{convolution inverse} of $\omega$ is defined to be the sequence $\omega^{(-1)}$ such that $\omega*\omega^{(-1)}=\omega^{(-1)}*\omega=\delta_{d}$. It is easy to see that $\omega^{(-1)}$ exists if and only if $\omega_0\ne 0$.

Given a sequence $v=(v_0, v_1, \ldots)$, its \emph{generating function $F_v$} is defined by 
$F_v(z)=\sum_{n=0}^{\infty} v_n z^n$  for $z\in\mathbb{C}$, 
where $\mathbb{C}$ denotes the complex plane.
This function should be  understood in the sense of analytic continuation: we choose the continuation that has the largest possible domain in the upper half-plane and is symmetric about the real axis. 
For example, the generating function of the sequence $(1, 1, ...)$ is given by $F_1(z):=\frac{1}{1-z}$, which is defined at all points in the  complex plane except $z=1$.
It is straightforward to verify that $F_{u*v}(z)=F_u(z)F_v(z)$ for any sequences $u$ and $v$. Hence, the generating functions of $\omega$ and $\omega^{(-1)}=:a$ are related by $F_a(z)=\frac{1}{F_{\omega}(z)}$.  

To solve the F-ODE \eqref{FODE}, we work on a uniform mesh $\{t_{n}\}_{n=0}^\infty$ defined by $t_n=nh$ with step size $h>0$.
Consider discretisations of the Caputo fractional derivative $\mathcal{D}_t^{\alpha} y(t_n)$ in the form                                                       
 \begin{equation}\label{eq:CMsch}
\begin{split}
\mathcal{D}_h^{\alpha} (y_{n}) =\frac{1}{h^{\alpha}}\sum\limits_{k=0}^{n}\omega_{k}(y_{n-k}-y_{0})
=\frac{1}{h^{\alpha}}\left(\delta_{n}y_{0} + \sum\limits_{k=0}^{n-1}\omega_{k}y_{n-k}\right)\ \text{ for }n=1,2,\dots
\end{split}
\end{equation}
where $\delta_{n}:=-\sum_{k=0}^{n-1}\omega_{k}$ and $y_k$ denotes the numerical approximation of $y(t_k)$ for each~$k$.
Schemes of this type clearly conserve mass, i.e., $\mathcal{D}_h^{\alpha} (y_{n}) =0$ for any constant sequence $\{y_n\}$.

A function $g:[0,\infty)\to \mathbb{R}$ is said to be \emph{completely monotone ($\mathcal{CM}$)} \cite[Section 3.1.2]{jin2021fractional}  if $(-1)^ng^{(n)}(x)\ge 0$ for all $x$ and $n$. Its discrete analogue: a sequence $v= (v_0,v_1, \dots)$ is said to be $\mathcal{CM}$ \cite[Section III.4]{Widder1941} if $((I-E)^jv)_k\ge 0$ for all $j,k\ge 0$, where $(Ev)_j:= v_{j+1}$.

Again setting $a:=\omega^{(-1)}$, we say (following \cite{LubSIAMJMA86}) that the discretisation \eqref{eq:CMsch}  is \emph{consistent} if  $h^{\alpha}F_a(e^{-h})=1+o(1)$ as $h\to 0^+$. The sequence $\omega$ corresponds to to the Caputo derivative $\mathcal{D}_t^{\alpha}$, so its convolution inverse~$a$ corresponds to the Riemann-Liouville fractional integral of order~$\alpha$. The kernel $k_{\alpha}(t) := t_+^{\alpha-1}/\Gamma(\alpha)$ of the Riemann-Liouville  integral is evidently a completely monotone function 
and we wish our discretisation \eqref{eq:CMsch} to inherit this property in order to preserve as much structure as possible. Thus, we say that a consistent discretisation \eqref{eq:CMsch} of $\mathcal{D}_t^{\alpha}$ is
\emph{$\mathcal{CM}$-preserving} if the sequence $a=\omega^{(-1)}$ is a $\mathcal{CM}$ sequence. See \cite{LiWang2019} for more information about $\mathcal{CM}$-preserving discretisations of $\mathcal{D}_t^{\alpha}$, which have several good properties.

Two particular examples of $\mathcal{CM}$-preserving schemes are the Gr\"{u}nwald-Letnikov and L1 discretisations.
For the Gr\"{u}nwald-Letnikov formula, the weights $\omega_{k}$ are the $k$-th coefficients of the generating function $F_\omega(z)=(1-z)^\alpha$; they can be computed iteratively from
$\omega_{0}=1,\ \omega_{k}=\left(1-\frac{\alpha+1}{k}\right)\omega_{k-1}$  for $k=1,2,\dots$.
For the  L1 scheme, one has 
\[
\omega_{0}=\frac{1}{\Gamma(2-\alpha)},\quad  
	\omega_{k}=\frac{1}{\Gamma(2-\alpha)}\left[(k+1)^{1-\alpha}-2k^{1-\alpha}+(k-1)^{1-\alpha}\right]\ \text{ for }k=1,2,\dots
\]
and consequently $\delta_{n}=-\sum\limits_{k=0}^{n-1}\omega_{k} = \left[(n-1)^{1-\alpha}-n^{1-\alpha}\right]/\Gamma(2-\alpha)$.

For both these schemes, the weights $\omega_{k}$ and $\delta_{n}$  have the convenient properties
\begin{equation}\label{eq:coeff}
\begin{cases}
      &\hbox{(i)} ~\omega_{0}>0, ~ \omega_{1}<\omega_{2}<\dots<\omega_{n}<0 \quad \hbox{(monotonicity)};\\
      &\hbox{(ii)}~ \sum\limits_{k=0}^{\infty}\omega_{k}= 0\quad \hbox{(conservation)}; \\
      &\hbox{(iii)}~ \omega_{n}=O(n^{-1-\alpha}), ~ \delta_{n}=-\sum\limits_{k=0}^{n-1}\omega_{k}= O(n^{-\alpha})
      	\quad \hbox{(uniform decay rate)}.\\
    \end{cases}
\end{equation}
These properties  imply that there exist constants $0<C_{3}\leq C_{4}$, which are independent of $n$, such that 
\begin{equation} \label{eq:coeffdecay}
\begin{split}
 \frac{C_{3}}{n^{1+\alpha}} \leq |\omega_{n}| \leq  \frac{C_{4}}{n^{1+\alpha}}\,,
 \quad  \frac{C_{3}}{n^{\alpha}} \leq |\delta_{n}| \leq  \frac{C_{4}}{n^{\alpha}}\,, \quad 
   \frac{C_{3}}{n^{\alpha}} \leq \sum_{k=n}^{\infty} |\omega_{k}| \leq  \frac{C_{4}}{ n^{\alpha}}\ \text{ for } n=1,2,\dots
\end{split}
\end{equation}

In fact, all $\mathcal{CM}$-preserving schemes satisfy~\eqref{eq:coeff} and~\eqref{eq:coeffdecay}; see \cite{LiLiu2018,LiWang2019}. These properties are our main assumptions in Section~\ref{sec:discFODE} when establishing the optimal long-time decay rate for numerical solutions of~\eqref{FODE}.  

\begin{remark}
When one moves to fractional schemes that are analogues of high-order (i.e., at least second-order) classical schemes such as BDF2, Crank-Nicolson, and the trapezoidal formula, then~\eqref{eq:coeff} and~\eqref{eq:coeffdecay} may no longer hold true \cite{LiWang2019}; in particular, the monotonicity property can be lost because $\omega_{k}>0$ for some $k\geq 1$, and consequently it becomes very difficult to establish energy inequalities such as Lemma~\ref{lem:keyineq} below.
\end{remark}

\section{Optimal decay rate for numerical solutions of F-ODEs}\label{sec:discFODE}
In this section we shall show (Theorem~\ref{thm:main}) that when the F-ODE \eqref{FODE} is solved numerically using a $\mathcal{CM}$-preserving scheme on the uniform mesh $t_n=nh$ for $n=0,1,\dots$, then the computed solution has exactly the same long-time decay rate as Lemma~\ref{lem:non-decay} showed for the exact solution of~\eqref{FODE}, i.e., its decay rate is optimal. This result will be proved analogously to the continuous case, viz., by constructing appropriate discrete subsolutions and supersolutions  that share the same long-time decay rate, then applying a discrete fractional comparison principle.

\begin{lemma}[Discrete fractional comparison principle] \cite[Proposition 2.3]{LiWang2019} \label{lem:Liwang21}
Let $\mathcal{D}_{h}^{\alpha} $ be the $\mathcal{CM}$-preserving discrete operator of \eqref{eq:CMsch}.
Let $f(\cdot)$ be nondecreasing.  Suppose that the sequences $\{u_{j}\}_{j=0}^{\infty}, \{y_{j}\}_{j=0}^{\infty},  \{v_{j}\}_{j=0}^{\infty}$ satisfy $u_0\leq y_0\le v_0$ and
 \begin{equation*} %\label{eq11}
\mathcal{D}_{h}^{\alpha} (u_{n})  +f(u_{n})\leq 0, \quad 
\mathcal{D}_{h}^{\alpha} (y_{n})  +f(y_{n})= 0, \quad 
0 \le\mathcal{D}_{h}^{\alpha} (v_{n}) +f(v_{n})  \ \text{ for }n\geq 1.
\end{equation*}
Then $u_{n}\leq y_n \le v_{n}$ for $n=0,1,2,\dots$. 
\end{lemma}

In Lemma~\ref{lem:Liwang21} we say that $\{u_n\}$ is a \emph{discrete subsolution for~$\{y_n\}$} and  $\{v_n\}$ is  a \emph{discrete supersolution for~$\{y_n\}$}. 

 Lemma~\ref{lem:Liwang21} is a generalisation of the comparison principle for the L1 method that was used in~\cite{Kop20,liao2020energy}. Note that the key properties needed to establish this lemma are the sign pattern and monotonicity of the discrete coefficients described in (\ref{eq:coeff}(i)).
 
Next, observe that  the discretisation \eqref{eq:CMsch} can be rewritten as
\begin{equation}\label{CMalt}
\mathcal{D}_h^{\alpha} (v_{n}) = \frac{1}{h^{\alpha}}\left[-\delta_1 v_{n}+\delta_{n}v_{0} 
		+ \sum\limits_{k=1}^{n-1}(\delta_k - \delta_{k+1}) v_{n-k}\right]
	= \frac{1}{h^{\alpha}}\sum\limits_{k=1}^{n} \delta_k (v_{n-k} - v_{n-k+1}) \text{ for } n\geq1.
\end{equation}
This reformulation resembles --- more closely than~\eqref{eq:CMsch} --- the standard definition 
$\frac{1}{\Gamma(1-\alpha)}\int_{0}^{t} (t-s)^{-\alpha}y'(s)\,ds$ of the Caputo derivative.  

We can now prove one of our main results: the discrete analogue of Lemma~\ref{lem:non-decay}.

\begin{theorem}\label{thm:main}
For the model equation \eqref{FODE}, consider the time-stepping scheme  defined by 
\begin{equation} \label{eq:numsolu}
\begin{split}
\mathcal{D}_{h}^{\alpha}(y_{n}) =\frac{1}{h^\alpha} \left( \sum_{k=1}^{n}\omega_{n-k}y_{k} +\delta_{n} y_{0}\right)=-\lambda y_{n}^{\gamma} \ \text{ for }n\geq 1, \ \text{with } y_0>0.
\end{split}
\end{equation}
where $\lambda, \gamma>0$ and $\mathcal{D}_{h}^{\alpha}$ is $\mathcal{CM}$-preserving.
Then there exists a positive constant $h_0$ such that for any $h$ satisfying $0<h\le h_0$, on the mesh $\{t_n = nh\}_{n=0}^\infty$ the solution $\{y_n\}$ of \eqref{eq:numsolu} satisfies
\begin{equation} \label{eq:numdecay}
 \frac{C_{5}}{1+t_{n}^{\alpha/\gamma}}\leq y_{n}\leq \frac{C_{6}}{1+t_{n}^{\alpha/\gamma}} \text{ for } n =0,1,2,\dots
\end{equation}
where the positive constants $C_{5}, C_{6}$ are independent of $n$ and $h$.
\end{theorem}

%%%

Before proving this theorem, we give a heuristic motivation for our construction of the discrete subsolution $\{u_{n}\}$ and supersolution $\{v_{n}\}$. For the numerical solution $\{y_{n}\}$, Theorem \ref{thm:main} requires the long-term decay rate 
\tcm{$y_{n}\sim t_{n}^{-\alpha/\gamma}$} to be consistent with the true solution (recall Lemma~\ref{lem:non-decay}), so we need \tcm{$u_{n}\sim  t_{n}^{-\alpha/\gamma}$ and $v_{n} \sim t_{n}^{-\alpha/\gamma}$ as $n\to\infty$.} We still have flexibility in choosing suitable constant multiplicative factors in $u_{n}$ and $v_{n}$ when $n$ is large, and in specifying suitable values of $u_{n}$ and $v_{n}$ when $n$ is relatively small, while keeping the sequences $\{u_{n}\}$ and $\{v_{n}\}$ monotonic decreasing and positive.  

\begin{proof}
Since $\mathcal{D}_{h}^{\alpha} (\cdot)$ is a $\mathcal{CM}$-preserving approximation, its weights $\{\omega_{j}\}_{j=0}^{\infty}$ satisfy \eqref{eq:coeff}  and~\eqref{eq:coeffdecay}. We shall construct a discrete subsolution $\{u_n\}$ and a discrete supersolution $\{v_n\}$  for $\{y_n\}$; the inequalities  \eqref{eq:numdecay} will then follow from Lemma~\ref{lem:Liwang21}. 

Define $h_0>0$ by 
\begin{equation}\label{eq:h0}
h_0 =  \frac12 \left(\frac{C_3 y_0^{1-\gamma}}{2\lambda}\right)^{1/\alpha}\,, 
\end{equation}
where the positive constant $C_3$ comes from~\eqref{eq:coeffdecay}. Assume throughout the proof that $0<h\le h_0$.

The subsolution $\{u_n\}$ is constructed first. 
Set $g_{n}= t_{n}^{\alpha}/\Gamma(1+\alpha)$ for $n\ge 0$. 
Define the sequence 
 \begin{equation}\label{eq:subsolu}
u_n= \begin{cases}
y_{0}-\mu g_{n} &\text{for } n=0,1, \dots, n_0,\\
C_7 t_{n}^{-\alpha/\gamma} &\text{for }n\ge n_0+1,
\end{cases}
\end{equation}
where the positive parameters $\mu,  n_0, C_7$ are specified by
\begin{equation}\label{eq:subcond}
\begin{split}
\mu:=\lambda y_{0}^{\gamma} \cdot \frac{\Gamma(1+\alpha)} {C_{3}}, \quad
t_{n_0} := \max\left\{ t_n: t_n^\alpha \le \frac{C_3 y_0^{1-\gamma}}{2 \lambda}\right\}, \quad
C_7 := \frac{ y_0 t_{n_{0}}^{\alpha/\gamma}} {2}\,.
\end{split}
\end{equation}
Note that $h\le h_0$ and \eqref{eq:h0} imply that $n_0\ge 2$ and $t_{n_0}\ge h_0$, so 
\begin{equation}\label{C7}
C_7\ge \frac{y_0}{2}\cdot \frac1{2^{\alpha/\gamma}} \left(\frac{C_3 y_0^{1-\gamma}}{2\lambda}\right)^{1/\gamma}
	 =  \frac1{2^{1+\alpha/\gamma}} \left(\frac{y_0C_3}{2\lambda}\right)^{1/\gamma},
\end{equation}
i.e., $C_7$ (which depends on the mesh) is bounded  below by a fixed positive constant.
From the definitions \eqref{eq:subcond} we get 
\[
 \frac{\mu t_{n_0}^{\alpha}} {\Gamma(1+\alpha)}+ C_7 t_{n_{0}+1}^{-\frac{\alpha}{\gamma}} \le \frac{y_0}{2} +  \frac{y_0}{2}  = y_0,
\]
which is equivalent to $u_{n_{0}}\geq u_{n_{0}+1}$. It then follows easily from~\eqref{eq:subsolu} that the sequence $\{u_{n}\}_{n=0}^{\infty}$ is monotonically decreasing and strictly positive. 

We now show that $\{u_{n}\}$ is a subsolution to $\{y_{n}\}$, considering separately the cases $1\leq n\leq n_{0}$ and $n>n_0$. 

\textbf{Case I:} $1\leq n\leq n_{0}$.
It follows from~\eqref{eq:subsolu} that 
\begin{equation*} 
\begin{split}
\mathcal{D}_h^{\alpha} (u_{n})&=\frac{1}{h^\alpha} \left( \sum_{k=1}^{n}\omega_{n-k}u_{k} +\delta_{n} u_{0}\right)\\
&=\frac{1}{h^\alpha} \left( \left(\sum_{k=1}^{n}\omega_{n-k} +\delta_{n}\right) u_{0}- \mu \sum_{k=1}^{n}\omega_{n-k}g_{k}\right)
=- \frac{\mu}{h^\alpha} \left( \sum_{k=1}^{n}\omega_{n-k}g_{k}\right),
\end{split}
\end{equation*}
since $\sum_{k=1}^{n}\omega_{n-k} +\delta_{n}=0$.
As the scheme is $\mathcal{CM}$-preserving, from \eqref{eq:coeff} and~\eqref{eq:coeffdecay} we have $\omega_0>0$, $\omega_k<0$ for $k\geq 1$
and $\omega_{0}=-\sum_{k=1}^{\infty}\omega_{k}$, which imply that  
\[
\sum_{k=1}^{n}\omega_{n-k}g_{k}=\sum_{k=1}^{n-1}\omega_{k}(g_{n-k}-g_{n})+ \left(-\sum_{k=n}^{\infty}\omega_{k}\right) g_{n}\geq  \left(-\sum_{k=n}^{\infty}\omega_{k}\right) g_{n}\geq \frac{C_{3}}{n^\alpha}g_{n}.
\]
Hence, recalling the definitions of $\mu$ and $g_n$, we get
\begin{equation} \label{1nn0}
\mathcal{D}_h^{\alpha} (u_{n}) \leq - \frac{\mu}{h^\alpha} \cdot \frac{C_{3}}{n^\alpha}\,g_{n} 
\leq -\lambda y_{0}^{\gamma} \leq - \lambda (u_{0}-\mu g_{n})^{\gamma} 
=  - \lambda u_n^{\gamma}  \ \text{ for } 1\le n\le n_0.
\end{equation}

\textbf{Case II:} $n\geq n_{0}+1$. 
Set 
$I_{1}:=\sum_{k=1}^{n_{0}}\omega_{n-k}u_{k} +\delta_{n} u_{0}$ and $I_{2}:=\sum_{k=n_{0}+1}^{n}\omega_{n-k}u_{k}.$
Then
\begin{equation*} %\label{eq12}
\begin{split}
I_{1}=\left(\delta_{n}+ \sum\limits_{k=1}^{n_{0}}\omega_{n-k}\right) u_{0}-\mu \sum\limits_{k=1}^{n_{0}}\omega_{n-k} g_{k}
=-\left(\sum\limits_{k=0}^{n-n_{0}-1}\omega_{k}\right) u_{0}-\mu \sum\limits_{k=1}^{n_{0}}\omega_{n-k} g_{k}.
\end{split}
\end{equation*}
Hence, recalling that $\omega_k<0$ for $k\ge 1$, we have
\begin{align*}
I_{1}+I_{2} &= \omega_{0}(u_{n}-u_{0}) + \sum_{k=1}^{n-n_{0}-1} |\omega_{k}| (u_{0}-u_{n-k}) 
	+\mu \sum_{k=1}^{n_{0}} |\omega_{n-k}| g_{k} \notag\\
&=\omega_{0}(u_{n}-u_{n_{0}}) + \sum_{k=1}^{n-n_{0}-1} |\omega_{k}| (u_{n_{0}}-u_{n-k} )+\mu \sum_{k=1}^{n_{0}} |\omega_{n-k}| g_{k}+\mu g_{n_{0}} \left( \sum_{k=1}^{n-n_{0}-1} |\omega_{k}| -\omega_{0} \right),  %\label{eq:mid1}
\end{align*}
using the identities $u_{0}= u_{n_{0}}-u_{n_{0}}+ u_{0} = u_{n_{0}}+\mu g_{n_{0}}$
and $u_{0}-u_{n-k}=u_{0}-u_{n_{0}}+u_{n_{0}}- u_{n-k}=\mu g_{n_{0}}+u_{n_{0}}-u_{n-k}$. 
Here the first two terms are nonpositive because $\omega_0>0$ and $\{u_n\}$ is monotonically decreasing, so 
 \begin{align*} 
I_{1}+I_{2} &\leq\mu \sum_{k=1}^{n_{0}} |\omega_{n-k}| g_{k}+\mu g_{n_{0}} \sum_{k=n-n_{0}}^{\infty} \omega_{k}\\
	&\leq \mu g_{n_{0}}\sum_{k=1}^{n_{0}} |\omega_{n-k}|+\mu g_{n_{0}}  \sum_{k=n-n_{0}}^{\infty} \omega_{k} 
	 = \mu g_{n_{0}} \sum_{k=n}^{\infty} \omega_{k} . 
\end{align*}
According to \eqref{eq:coeffdecay} and the definitions of~$\mu$ and $g_{n_0}$, we get that 
$I_{1}+I_{2}\leq - \mu g_{n_{0}} \frac{ C_{3}} { n^{\alpha}}=- \frac{\lambda y_{0}^{\gamma} t_{n_{0}}^{\alpha}} {n^{\alpha}}$.
Thus
\begin{equation}\label{n0lessthann}
\mathcal{D}_h^{\alpha} (u_{n}) = \frac1{h^\alpha}(I_{1}+I_{2})
	\le - \frac{\lambda y_{0}^{\gamma} t_{n_{0}}^{\alpha}} {h^\alpha n^{\alpha}}
	=  - \lambda y_{0}^{\gamma} t_{n_{0}}^{\alpha} t_n^{-\alpha}
	\le -\lambda u_n^\gamma  \ \text{ for } n> n_0.
\end{equation}
from the definitions of $u_n$ and $C_7$.

The inequalities \eqref{1nn0} and~\eqref{n0lessthann}, together with $u_0 = y_0$,  
show that  $\{u_n\}$  is a discrete subsolution for~$\{y_n\}$.\\

We now construct a discrete supersolution $\{v_n\}$ for $\{y_n\}$. 
Define  the mesh point 
\begin{equation}\label{tN}
t_{n_1} := \min \left\{ t_n:  t_{n}^\alpha \ge \frac{y_0^{1-\gamma}}{\lambda}\max\left\{  \frac{\alpha 2^\alpha C_4 }{\gamma(1-\alpha)}\,,
	C_8 \right\} \right\},
\end{equation}
where, recalling the positive constant $C_4$ in~\eqref{eq:coeffdecay}, we define
$
C_8 :=  2^\alpha C_4 \left[ \frac{\alpha 2^{\alpha/\gamma} }{\gamma(1-\alpha)} + 1 \right].
$
Then define the function
\begin{equation}\label{w}
v(t) := 
	\begin{cases}
	y_0  &\text{for } 0\le t\le t_{n_1}, \\
	C_9 t^{-\alpha/\gamma} &\text{for } t_{n_1} < t < \infty,
	\end{cases}
\ \text{with }\ C_9:=  y_0 t_{n_1}^{\alpha/\gamma}.
\end{equation}
The choice of $C_9$ ensures that $v\in C[0,\infty)$. Clearly $v$ is positive and monotonically decreasing on $[0,\infty)$. Set $v_n := v(t_n)$ for $n=1,2,\dots$

From \eqref{eq:CMsch} we have 
\begin{equation}\label{w11}
\mathcal{D}_h^{\alpha}(v_n)  =0 > -\lambda v_n^\gamma \ \text{ for } n=1,2,\dots, {n_1}.
\end{equation}

Next, suppose that ${n_1} <n \le 2{n_1}$. 
Recalling~\eqref{CMalt} and the definition of $v(t)$, since $\delta_k <0$ for $k\ge 1$ one gets
\[
\mathcal{D}_h^{\alpha} (v_{n}) =  h^{-\alpha}\sum_{k=1}^{n-{n_1}} \delta_k (v_{n-k} - v_{n-k+1}) \notag\\
	\ge h^{-\alpha}(v_{n_1} - v_{{n_1}+1}) \sum_{k=1}^{n-{n_1}} \delta_k;
\]
by  \eqref{eq:coeffdecay} and $n\le 2{n_1}$ one has
\[
\sum_{k=1}^{n-{n_1}} |\delta_k| \le \sum_{k=1}^{n-{n_1}} C_4 k^{-\alpha} \le   C_4 \int_{s=0}^{n-{n_1}}s^{-\alpha}ds 
	= \frac{C_4 (n-{n_1})^{1-\alpha}}{1-\alpha} \le \frac{C_4 {n_1}^{1-\alpha}}{1-\alpha},
\]
so 
\begin{equation}\label{discsup1}
\mathcal{D}_h^{\alpha} (v_{n}) \ge  - \frac{C_4 }{1-\alpha}h^{-\alpha}(v_{n_1} - v_{{n_1}+1}) {n_1}^{1-\alpha}.
\end{equation}
But $v'<0$ and $v''>0$ on $(t_{n_1}, t_{{n_1}+1})$ imply that
$
0 < v_{n_1} - v_{{n_1}+1}\le -hv'(t_{n_1}) = h\frac{C_9\alpha}{\gamma}t_{n_1}^{-1-\alpha/\gamma} 
	= hy_0 \frac{\alpha}{\gamma}t_{n_1}^{-1} 
$
by the definition of $C_9$. Hence \eqref{discsup1} yields
\begin{equation}\label{discsup2}
\mathcal{D}_h^{\alpha} (v_n)   \ge - \frac{y_0 \alpha C_4 }{\gamma(1-\alpha)} t_{n_1}^{-\alpha}
	\ge - \frac{y_0 \alpha 2^\alpha C_4 }{\gamma(1-\alpha)} t_n^{-\alpha}\ \text{ since }t_n \le 2t_{n_1}.
%	= - \lambda v_n^\gamma \ \text{ for } {n_1}<n\le 2{n_1}. 
\end{equation}
Now the definitions of $t_{n_1}$ and $C_9$ give
$ \frac{\alpha 2^\alpha C_4 }{\gamma(1-\alpha)} \le \lambda y_0^{\gamma-1}t_{n_1}^\alpha
 	= \lambda C_9^\gamma y_0^{-1}$,
so from \eqref{discsup2} we get the desired inequality
\begin{equation}\label{discsup3}
\mathcal{D}_h^{\alpha} (v_n)   \ge -  \lambda C_9^\gamma  t_n^{-\alpha} 
		= - \lambda v_n^\gamma \ \text{ for } {n_1}<n\le 2{n_1}. 
\end{equation}

Finally, suppose that  $n > 2{n_1}$. Set $n' := \lceil n/2\rceil$, so $n' >{n_1}$ and $n'\le (n+1)/2$.
Using \eqref{CMalt}, the definition of $v(t)$ and \eqref{eq:coeffdecay}, one has
\begin{align}
\mathcal{D}_h^{\alpha} (v_{n}) &=  h^{-\alpha}\sum_{k=1}^{n-{n_1}} \delta_k (v_{n-k} - v_{n-k+1}) \notag\\
	&\ge -h^{-\alpha}\left(\sum_{k=1}^{n-n'} +  \sum_{k=n-n'+1}^{n-{n_1}} \right) C_4 k^{-\alpha}(v_{n-k} - v_{n-k+1})   \notag\\
	&\ge -C_4 h^{-\alpha}\left[ (v_{n'} - v_{n'+1}) \sum_{k=1}^{n-n'} k^{-\alpha}
		+  (n-n'+1)^{-\alpha} \sum_{k=n-n'+1}^{n-{n_1}} (v_{n-k} - v_{n-k+1})\right]   \notag\\
	&\ge -C_4 h^{-\alpha}\left[ h|v'(t_{n'})| \int_{s=0}^{n-n'} s^{-\alpha}ds 
		+  \left(\frac{n+1}{2} \right)^{-\alpha} (v_{n_1} - v_{n'})\right].   \label{discsup4}
\end{align}
But the definitions of $v(t)$ and $C_9$, and $n'\ge n/2$, give
\[
|v'(t_{n'})|  = \frac{C_9\alpha}{\gamma}t_{n'}^{-1-\alpha/\gamma} 
	\le y_0 t_{n_1}^{\alpha/\gamma}  \cdot \frac{\alpha}{\gamma}\left( \frac{t_n}{2} \right)^{-1-\alpha/\gamma} 
	< y_0 \frac{\alpha}{\gamma} \,2^{1+\alpha/\gamma}  t_n^{-1} \ \text{ as }t_{n_1} < t_n.
\]
Substituting this inequality into~\eqref{discsup4}, recalling that  $v_{n_1}=y_0$ and discarding the final term $v_{n'}$, we get
\begin{align}
\mathcal{D}_h^{\alpha} (v_{n}) &\ge -C_4  y_0 h^{-\alpha} \left[ \frac{\alpha}{\gamma} \,2^{1+\alpha/\gamma} 
		h t_n^{-1} \frac{(n-n')^{1-\alpha}}{1-\alpha} + 2^\alpha n^{-\alpha} \right].  \notag\\
	&\ge -C_4  y_0 2^\alpha t_n^{-\alpha} \left[ \frac{\alpha 2^{\alpha/\gamma} }{\gamma(1-\alpha)} + 1 \right],  
		 \label{discsup5}
\end{align}
where we used  $ht_n^{-1} = n^{-1},\, n-n'\le n/2$ and $h^{-\alpha} n^{-\alpha}  = t_n^{-\alpha}$. The definitions of $C_8, t_{n_1}, C_9$ and $v(t)$ enable us to deduce from \eqref{discsup5} that 
\begin{equation}\label{discsup6}
\mathcal{D}_h^{\alpha} (v_{n}) 	\ge -C_8  y_0  t_n^{-\alpha}  \ge -\lambda y_0^\gamma t_{n_1}^{\alpha} t_n^{-\alpha}
	= -\lambda C_9^\gamma  t_n^{-\alpha}  = -\lambda v_n^\gamma \ \text{ for } n >2{n_1}.	 
\end{equation}

Combining \eqref{w11}, \eqref{discsup3} and \eqref{discsup6} gives $\mathcal{D}_h^{\alpha} (v_n) \ge - \lambda v_n^\gamma $ for all $n\ge 1$; as $v_0 = y_0$, we have shown that $\{v_n\}$ is a supersolution for $\{y_n\}$.

It is evident that $u_n = C_7 t_{n}^{-\alpha/\gamma}$ and $v_n \le y_0(1+t_n^{-\alpha/\gamma})$ for $t_n\ge \max\{t_{n_0}, t_{n_1}\}$ , so the desired bounds \eqref{eq:numdecay}  now follow immediately from Lemma~\ref{lem:Liwang21} on recalling \eqref{C7}.  
\end{proof}

\begin{remark}
The discrete subsolution and discrete supersolution constructed in the proof of Theorem~\ref{thm:main} are discrete analogues of  
the subsolution and supersolution for the continuous problem \eqref{FODE}  in \cite[Section~7]{vergara2015optimal}. This agreement is unsurprising since these functions  are natural choices to act as sub- and supersolutions.  Note however that  the discrete analysis above is more complicated than the continuous analysis  in \cite[Section~7]{vergara2015optimal}. 
\end{remark}

\begin{remark}
In the linear case $\gamma=1$,  the long-time numerical decay rate $y_{n}=O(t_n^{-\alpha})$ was proved in~\cite{wang2021mittag} via a singularity analysis of the generating function. 
Theorem~\ref{thm:main} now gives an alternative proof of this result based on discrete upper and lower solutions and the fractional comparison principle.
\end{remark}

\section{The L1 scheme on general meshes}\label{sec:L1}
Theorem~\ref{thm:main} is applicable to all $\mathcal{CM}$-preserving schemes on uniform meshes. In the present section we consider a particular $\mathcal{CM}$-preserving scheme ---  the L1 scheme --- and extend Theorem~\ref{thm:main} to this scheme on very general meshes. 
 
 Working on (possibly) nonuniform meshes implies that the bounds \eqref{eq:coeffdecay} can no longer be assumed, but on the other hand the L1 scheme has the special property \eqref{L1defn} which will be used several times in our analysis.

Let the mesh $0=t_0 <t_1 <t_2 <\dots$ be arbitrary. Set $\tau_n = t_n-t_{n-1}$ for $n\ge 1$. Let $y_n$ denote the computed solution at each mesh point $t_n$ for $n\ge 1$. Then the L1 discretisation~$\mathcal{D}_{L1}^{\alpha}$ of $\mathcal{D}_t^{\alpha}$ can be written as (see for example \cite{StyL1survey})
\begin{equation}\label{eq:L1a}
\mathcal{D}_{L1}^{\alpha} (y_n) = \frac1{\Gamma(2-\alpha)} \left[ d_{n,1} y_n - d_{n,n}y_0 
	+ \sum^{n-1}_{k=1}y_{n-k} \left(d_{n,k+1}-d_{n,k}\right)  \right] ,
\end{equation}
where 
\[  
d_{n,k} :=\frac{(t_n-t_{n-k})^{1-\alpha}-(t_n-t_{n-k+1})^{1-\alpha}}{\tau_{n-k+1}}
	= \frac{1-\alpha}{\tau_{n-k+1}} \int_{s=t_{n-k}}^{t_{n-k+1}} (t_n-s)^{-\alpha} ds.
\]
Clearly $d_{n,k}$ is $1-\alpha$ times the mean value of the function $s\mapsto (t_n-s)^{-\alpha}$ on the interval $[t_{n-k}, t_{n-k+1}]$; as this function is positive and increasing, it follows that $0 < d_{n,k+1} < d_{n,k}$ for all $k$ and $n$.

Thus, our computed discrete solution $\{y_n\}_{n=0}^\infty$ of~\eqref{FODE} is defined by
\begin{equation}\label{L1scheme}
\mathcal{D}_{L1}^{\alpha} (y_n) = -\lambda y_n^\gamma\ \text{ for }n=1,2,\dots, \ \text{ with }y_0 \text{ given in }\eqref{FODE}.
\end{equation}
We aim to derive a decay result for $\{y_n\}$ that is similar to Theorem~\ref{thm:main}.

Our analysis will use a defining property of the L1 discretisation: if $\phi\in W^{1,1}_{loc}(\mathbb{R}^+)$ and $\phi_n := \phi(t_n)$ for all~$n$, then 
\begin{equation}\label{L1defn}
\mathcal{D}_{L1}^{\alpha}(\phi_n) = \mathcal{D}_t^{\alpha} \phi_I(t_n)\ \text{ for } n=1,2,\dots, 
\end{equation}
where $\phi_I$ is the piecewise linear interpolant to $\phi$ on the mesh, 
i.e., $\phi_I(t)=\frac{t-t_{n-1}}{t_{n}-t_{n-1}} \phi_n+ \frac{t_{n}-t}{t_{n}-t_{n-1}} \phi_{n-1}$ for $t\in [t_{n-1}, t_{n}]$.

Note that Lemma~\ref{lem:Liwang21} remains valid when $\mathcal{D}_h^{\alpha}$ on a uniform mesh is replaced by $\mathcal{D}_{L1}^{\alpha}$ on an arbitrary mesh, since its proof requires only the scheme properties (in the L1 scheme notation) $0 < d_{n,k+1} < d_{n,k}$ for all~$k$ and~$n$, and $\mathcal{D}_{L1}^{\alpha}(z_n)=0$ for any constant sequence~$\{z_n\}$. We shall use this lemma in our analysis.

The construction of the subsolution and supersolution used in Section~\ref{sec:L1} is related to  Section~\ref{sec:discFODE}, but there are some significant differences because \eqref{eq:coeffdecay}  is no longer available. 

\subsection{Discrete subsolution}\label{sec:discsub}
Define
\begin{equation}\label{mu}
\mu := \lambda y_0^\gamma \max\left\{1, 2\left( \frac{3}{4} \right)^\gamma \Gamma(1+\alpha) \Gamma(1-\alpha) \right\}.
\end{equation}
Here $ \Gamma(1+\alpha) \Gamma(1-\alpha) = \alpha  \Gamma(\alpha) \Gamma(1-\alpha) = \alpha\pi/\sin(\alpha\pi) >1$ by \cite[Theorem D.3]{Diethelm2010}.
Assume that  some mesh point $\hat t$  satisfies 
\begin{equation}\label{hatt}
\left( \frac{y_0\Gamma(1+\alpha)}{4\mu} \right)^{1/\alpha}  \le \hat t \le \left( \frac{y_0\Gamma(1+\alpha)}{2\mu} \right)^{1/\alpha}.
\end{equation}
(One can ensure that \eqref{hatt} is satisfied by sufficient mesh refinement near $t^\alpha = y_0\Gamma(1+\alpha)/(2\mu)$; this is not restrictive.)
For any constant $\beta>0$, define $g_\beta(s) = s^{\beta-1}/\Gamma(\beta)$ for $s\in (0,\infty)$, (if $\beta>1$ we also define $g_\beta(0)=0$).
Then define the function
\begin{equation}\label{v}
v(t) := 
	\begin{cases}
	y_0 - \mu g_{1+\alpha}(t) &\text{for } 0\le t\le \hat t, \\
	C_1^{\prime} t^{-\alpha/\gamma} &\text{for } \hat t < t < \infty,
	\end{cases}
\ \text{with }\ C_1^{\prime}:=  \hat t^{\alpha/\gamma} \left( y_0 - \mu g_{1+\alpha}(\hat t)\right).
\end{equation}
The choice of $C_1^{\prime}$ ensures that $v\in C[0,\infty)$. 
Note that $v(\hat t) = y_0 - \mu g_{1+\alpha}\left(\hat t \right)$ and \eqref{hatt} imply that 
\begin{equation}\label{vthat}
\frac{y_0}{2} \le v(\hat t) \le \frac{3y_0}{4}\,;
\end{equation}  
it follows that $v(t)>0$ for $t\in [0,\infty)$ and $C_1' \ge \frac{y_0}{2}\left( \frac{y_0\Gamma(1+\alpha)}{4\mu} \right)^{1/\gamma}$. Furthermore, clearly $v$ is a strictly decreasing function.

Set $v_n = v(t_n)$ for $n=0,1,\dots$ For $0< t_k \le t_n \le \hat t$, by Chebyshev's integral inequality \cite[Theorem 43]{HLP88},  
since $s\mapsto (t_n-s)^{-\alpha}$ is increasing while $s\mapsto g_{1+\alpha}'(s)$ is decreasing, one has
\begin{align*}
\int_{s=t_{k-1}}^{t_k} (t_n-s)^{-\alpha} g_{1+\alpha}'(s)\,ds 
	&\le \frac1{t_k - t_{k-1}} \left( \int_{s=t_{k-1}}^{t_k} (t_n-s)^{-\alpha} \,ds \right)
		\left( \int_{s=t_{k-1}}^{t_k}  g_{1+\alpha}'(s)\,ds \right) \\
	&= \frac{g_{1+\alpha}(t_k) - g_{1+\alpha}(t_{k-1})}{t_k - t_{k-1}}  \int_{s=t_{k-1}}^{t_k} (t_n-s)^{-\alpha} \,ds \\
	&=  \int_{s=t_{k-1}}^{t_k} (t_n-s)^{-\alpha}  (g_{1+\alpha})_I'(s) \,ds.
\end{align*}
Hence, recalling \eqref{L1defn}, we get
\begin{align*}
\mathcal{D}_{L1}^{\alpha} g_{1+\alpha}(t_n) &= \mathcal{D}_t^{\alpha} (g_{1+\alpha})_I(t_n) 
		= \frac1{\Gamma(1-\alpha)}\int_{s=0}^{t_n} (t_n-s)^{-\alpha} (g_{1+\alpha})_I'(s)\,ds \\
	&\ge \frac1{\Gamma(1-\alpha)}\int_{s=0}^{t_n} (t_n-s)^{-\alpha} g_{1+\alpha}'(s)\,ds	 
	= \mathcal{D}_t^{\alpha} g^{1+\alpha}(t_n) 
	= 1
\end{align*}
from \cite[p.193]{Diethelm2010}. Consequently
\begin{equation}\label{vtn1}
\mathcal{D}_{L1}^{\alpha} (v_n) \le -\mu \le -\lambda y_0^\gamma \le -\lambda v_n^\gamma\ \text{ for } 0<t_n<\hat t
\end{equation}
by \eqref{mu} and the definition of~$v$.

Now suppose that $t_n > \hat t$. Then 
$\mathcal{D}_{L1}^{\alpha} (v_n) = \mathcal{D}_t^{\alpha} v_I(t_n) = \frac1{\Gamma(1-\alpha)}\int_{s=0}^{t_n} (t_n-s)^{-\alpha} v_I'(s)\,ds$, and  since $v$ is decreasing one has 
\[
\int_{s=0}^{t_n} (t_n-s)^{-\alpha} v_I'(s)\,ds \le \int_{s=0}^{\hat t} (t_n-s)^{-\alpha} v_I'(s)\,ds  
	\le t_n^{-\alpha} \int_{s=0}^{\hat t}  v_I'(s)\,ds 
	= t_n^{-\alpha} \left[v(\hat t)-v(0)\right]  	\le -\frac{t_n^{-\alpha}y_0}{4}
\]
by \eqref{vthat}. Hence for $t_n > \hat t$ we get
\begin{align}
\mathcal{D}_{L1}^{\alpha} v_n + \lambda v_n^\gamma 
	&\le t_n^{-\alpha} \left[  -\frac{y_0}{4\Gamma(1-\alpha)} + \lambda C_1^\gamma \right] \notag\\
	&= t_n^{-\alpha} \left[  -\frac{y_0}{4\Gamma(1-\alpha)} + \lambda \hat t^\alpha \left( \frac{3y_0}{4} \right)^\gamma \right] \notag\\
	&\le t_n^{-\alpha} \left[  -\frac{y_0}{4\Gamma(1-\alpha)} 
		+ \lambda \frac{y_0\Gamma(1+\alpha)}{2\mu}\left( \frac{3y_0}{4} \right)^\gamma \right] \notag\\
	&\le 0,  \label{vtn2}
\end{align}
where we used \eqref{v}, \eqref{vthat} and finally \eqref{hatt} and the definition \eqref{mu} of~$\mu$.

From \eqref{vtn1} and \eqref{vtn2} we see that $\mathcal{D}_{L1}^{\alpha} v_n \le - \lambda v_n^\gamma $ for all $t_n>0$; as $v_0 = y_0$, we have shown that $\{v_n\}$ is a subsolution for $\{y_n\}$.

\subsection{Discrete supersolution}\label{sec:discsuper}
In this section, unlike Section~\ref{sec:discsub} where the mesh was arbitrary apart from the requirement~\eqref{hatt}, we  impose the following mild condition.
\begin{assumption}\label{ass:meshK}
Assume that the mesh $0=t_0 <t_1 <t_2 < \dots$satisfies the condition 
\[
\max_{n\ge 2} \left\{ \frac{\tau_n}{\tau_{n-1}}\,,  \frac{\tau_{n-1}}{\tau_{n}} \right\} \le K
	\  \text{for some fixed constant }K\ge 1.
\]
\end{assumption}

This assumption implies that for each $n\ge 2$, some mesh point is a crude approximation of  $t_n/2$, in the sense of the following lemma. 

\begin{lemma}\label{lem:meshK}
For each~$t_n$ with $n\ge 2$, there is a mesh point $t_{n^*}$ with the property 
\[  
 \frac{t_n}{K+2} \le t_{n^*}  \le  \frac{(K+1)t_n}{K+2}. 
\]  
\end{lemma}
\begin{proof}
Suppose that the result is false. Then we can choose $t_n$ with $n\ge 2$ such that the interval 
$[t_n/(K+2), \, (K+1)t_n/(K+2)]$ contains no mesh point. Since $n\ge 2$, at least one of the intervals $(0, t_n/(K+2))$ 
and $((K+1)t_n/(K+2), t_n)$ contains a mesh point; let us say it is  $(0, t_n/(K+2))$ as by symmetry the other case is similar. 
Set $t_m = \max_j \{t_j: 0 < t_j < t_n/(K+2)\}$. Then we must have $t_{m+1} \in ((K+1)t_n/(K+2), t_n]$. Consequently
$\tau_m = t_m - t_{m-1} \le t_m < t_n/(K+2)$ and
\[
\tau_{m+1} = t_{m+1}-t_m > \frac{(K+1)t_n}{K+2} - \frac{t_n}{K+2} = \frac{Kt_n}{K+2}\,,
\]
which together imply that $\tau_{m+1}/\tau_m > K$, contradicting our hypothesis.
\end{proof}

We now construct our supersolution.
Define $t_N>0$ to be the smallest mesh point satisfying 
\begin{equation}\label{tN}
t_N^\alpha \ge \frac{y_0^{1-\gamma}}{\lambda}\max\left\{ \frac{\alpha  (K+1)^{1-\alpha}(K+2)^\alpha}{\gamma\Gamma(2-\alpha)}\,,
	C_2^{\prime} \right\},
\end{equation}
where 
\[ 
C_2^{\prime} :=  \frac{ \left(K+2\right)^\alpha}{\Gamma(1-\alpha)} 
		\left\{ 1+ \frac{\alpha \left(K+1\right)^{1-\alpha} \left(K+2\right)^{\alpha/\gamma}}{\gamma(1-\alpha)} \right\}.
\]
%%%
Then define the function
\begin{equation}\label{w}
w(t) := 
	\begin{cases}
	y_0  &\text{for } 0\le t\le t_N, \\
	C_3^{\prime} t^{-\alpha/\gamma} &\text{for } t_N < t < \infty,
	\end{cases}
\ \text{with }\ C_3^{\prime}:=  y_0 t_N^{\alpha/\gamma}.
\end{equation}
Clearly $w\in C[0,\infty)$ is positive and monotonically decreasing. Set $w_n := w(t_n)$ for $n=1,2,\dots$

By \eqref{L1defn} we have 
\begin{equation}\label{w1}
\mathcal{D}_{L1}^{\alpha} w_n = \mathcal{D}_t^{\alpha} w_I(t_n) =0 > -\lambda w_n\ \text{ for } n=1,2,\dots, N.
\end{equation}

Next, suppose that $t_n \in (t_N, (K+2) t_N]$. Set $w_I'(t_N^+) := \lim_{s\to t_N+0} w_I'(s)$. Note that $w_I'(s)$ is negative and monotonically increasing almost everywhere on $(t_N, \infty)$. Again using \eqref{L1defn}, we get
\begin{align}
\mathcal{D}_{L1}^{\alpha} w_n &= \frac1{\Gamma(1-\alpha)} \int_{s=0}^{t_n} (t_n-s)^{-\alpha} w_I'(s)\,ds \notag\\
	&=  \frac1{\Gamma(1-\alpha)} \int_{s=t_N}^{t_n} (t_n-s)^{-\alpha} w_I'(s)\,ds \notag\\
	&\ge  \frac{w_I'(t_N^+)}{\Gamma(1-\alpha)} \int_{s=t_N}^{t_n} (t_n-s)^{-\alpha}\,ds \notag\\
	&= \frac{w_I'(t_N^+)}{\Gamma(2-\alpha)} (t_n-t_N)^{1-\alpha} \notag\\
	&\ge \frac{w_I'(t_N^+)\left((K+1) t_N\right)^{1-\alpha}}{\Gamma(2-\alpha)}\,,  \label{w2}
\end{align}
as $t_n\le (K+2) t_N$. Now $0 > w_I'(t_N^+) \ge w'(t_N^+) = -C_3^{\prime} (\alpha/\gamma)t_N^{-1-\alpha/\gamma}$ 
because $w''>0$ on $(t_N, t_{N+1})$, so \eqref{w2} yields
\begin{align}
\mathcal{D}_{L1}^{\alpha} w_n &\ge   - \frac{C_3^{\prime} \alpha (K+1)^{1-\alpha}}{\gamma\Gamma(2-\alpha)}\, t_N^{-\alpha-\alpha/\gamma} \notag\\
	&=   - \frac{y_0 \alpha (K+1)^{1-\alpha}}{\gamma\Gamma(2-\alpha)}\, t_N^{-\alpha}  \notag\\
	&\ge - \frac{y_0 \alpha (K+1)^{1-\alpha}}{\gamma\Gamma(2-\alpha)}\, \left(\frac{t_n}{K+2}\right)^{-\alpha},   \label{w3} 
\end{align}
where we used $C_3^{\prime}=  y_0 t_N^{\alpha/\gamma}$ and  $t_n \le (K+2)t_N$.
But
\[
\frac{\alpha  (K+1)^{1-\alpha}(K+2)^\alpha}{\gamma\Gamma(2-\alpha)} \le \lambda y_0^{\gamma-1}t_N^\alpha = \lambda C_3^{\prime \gamma} y_0^{-1}
\]
by \eqref{tN} and \eqref{w}; this inequality and \eqref{w3} give 
\begin{equation}\label{w4}
\mathcal{D}_{L1}^{\alpha} w_n   \ge -\lambda C_3^{\prime \gamma} t_n^{-\alpha} = - \lambda w_n^\gamma \ \text{ for } t_n \in (t_N, (K+2)t_N]. 
\end{equation}

Finally, suppose that  $t_n > (K+2)t_N$. This implies that $n\ge 2$.
Then by \eqref{L1defn} and Lemma~\ref{lem:meshK}, we have
\begin{align*}
\mathcal{D}_{L1}^{\alpha} w_n &=  \frac1{\Gamma(1-\alpha)} \int_{s=t_N}^{t_n} (t_n-s)^{-\alpha} w_I'(s)\,ds \notag\\
	&= \frac1{\Gamma(1-\alpha)} \left(\int_{s=t_N}^{t_{n^*}} + \int_{s=t_{n^*}}^{t_n}\right)(t_n-s)^{-\alpha} w_I'(s)\,ds 
\end{align*}
--- note that $t_{n^*} > t_N$ because $t_n > (K+2)t_N$. Hence
\begin{align}
\mathcal{D}_{L1}^{\alpha} w_n &\ge \frac1{\Gamma(1-\alpha)} \left[ (t_n-t_{n^*})^{-\alpha} \int_{s=t_N}^{t_{n^*}}  w_I'(s)\,ds
		+  w_I'(t_{n^*}^+) \int_{s=t_{n^*}}^{t_n}(t_n-s)^{-\alpha} \,ds\right]  \notag\\
	&=  \frac1{\Gamma(1-\alpha)} \left\{ (t_n-t_{n^*})^{-\alpha} \left[w(t_{n^*})-w(t_N)\right]
		+  w_I'(t_{n^*}^+) \frac{(t_n-t_{n^*})^{1-\alpha}}{1-\alpha}   \right\}  \notag\\
	&>   \frac1{\Gamma(1-\alpha)} \left\{ \left( \frac{t_n}{K+2} \right)^{-\alpha} (-y_0)
		+  w'(t_{n^*}) \left(\frac{K+1}{K+2}\right)^{1-\alpha} \frac{t_n^{1-\alpha}}{1-\alpha} \right\} \label{w5}
\end{align}
by Lemma~\ref{lem:meshK} and $w'(t_{n^*})  < w_I'(t_{n^*}^+)$ since $w''>0$ on $(t_N,\infty)$. But
\[
w'(t_{n^*}) = -C_3^{\prime} \frac{\alpha}{\gamma}  t_{n^*}^{-\frac{\alpha}{\gamma}-1} 
	\ge - y_0 t_N^{\alpha/\gamma}\, \frac{\alpha}{\gamma} \left(\frac{t_n}{K+2}\right)^{-\frac{\alpha}{\gamma}-1}
	>  - y_0  \frac{\alpha}{\gamma} \left(\frac1{K+2}\right)^{-\frac{\alpha}{\gamma}-1}t_n^{-1}
\]
by the definition of $C_3^{\prime}$, another appeal to Lemma~\ref{lem:meshK}, and $t_N < t_n$.
Substituting this inequality into~\eqref{w5} yields
\begin{align}
\mathcal{D}_{L1}^{\alpha} w_n 
	&>   -\frac{y_0 t_n^{-\alpha}}{\Gamma(1-\alpha)} \left\{ \left(K+2\right)^\alpha
		+\frac{\alpha}{\gamma} \left(K+2\right)^{\frac{\alpha}{\gamma}+1} 
			\left(\frac{K+1}{K+2}\right)^{1-\alpha} \frac1{1-\alpha} \right\}  \notag\\
	&= - C_2^{\prime} y_0  t_n^{-\alpha}   \ge - \lambda w_n^\gamma   \label{w6}
\end{align}
since $\lambda C_3^{\prime \gamma} = \lambda y_0^\gamma t_N^\alpha \ge C_2^{\prime} y_0$ by the definition~\eqref{tN} of~$t_N$.

From \eqref{w1}, \eqref{w4} and \eqref{w6}, we see that $\mathcal{D}_{L1}^{\alpha} w_n \ge - \lambda w_n^\gamma $ for all $t_n>0$. Recall that $w_0 = y_0$. Thus, $\{w_n\}$ is a supersolution for $\{y_n\}$.

We have now proved the following optimal decay rate result.

\begin{theorem}\label{thm:L1}
Let the mesh satisfy \eqref{hatt} and Assumption~\ref{ass:meshK}.
When the L1 scheme \eqref{L1scheme} is used to solve \eqref{FODE}, there exist positive constants $C_{5}'$ and $C_{6}'$, which are independent of $n$ and the mesh,  such that 
\[
 \frac{C_{5}'}{1+t_{n}^{\alpha/\gamma}}\leq y_{n}\leq \frac{C_{6}'}{1+t_{n}^{\alpha/\gamma}} \text{ for } n =0,1,2,\dots
\]
\end{theorem}

\section{Optimal decay rate for numerical solutions of F-PDEs}\label{sec:discFPDE}
Consider F-PDEs of the form~\eqref{eq:FPDEs} that satisfy the structural condition \eqref{eq:FPstr}.
In this section we shall how the results of Section~\ref{sec:discFODE} for F-ODEs can be extended to this class of F-PDEs.

Consider semi-discretisations of~\eqref{eq:FPDEs} where the fractional derivative $\mathcal{D}_{t}^{\alpha}u$ is discretised on a uniform temporal mesh $\{t_n: t_{n}=nh\ \text{ for } n=0,1,2,\dots\}$   by a $\mathcal{CM}$-preserving scheme. Writing $U^n=U^{n}(x)$ for the approximation of $u(\cdot, t_n)$ at each $t_n$, we have 
\begin{equation}\label{eq:disFPDEs}
\begin{cases}
&\mathcal{D}_{h}^{\alpha}(U^{n})+\mathcal{N}[U^{n}](x)=0 \  \text{ for } x\in\Omega \text{ and } n=1,2,\dots, \\
&U^{n}(x)=0\ \text{ for }x\in\mathbb{R}^{d}\setminus \Omega\text{ and } n=1,2,\dots, \\
&U^{0}(x)=u_0(x)\ \text{ for } x\in\Omega.
\end{cases}
\end{equation}
As for \eqref{eq:FPDEs} in Section~\ref{sec:timefracprob}, we assume that $U^{0}\ge 0$, $U^{0}\not\equiv 0$, and $U^0\in L^{p}(\Omega)$.  
We also assume that for each $n\ge 1$ the discrete solution $U^{n}\not\equiv 0$, \tcm{$U^n\in L^{s}(\Omega)$ for some $s\in(1, \infty)$} 
and $U^n$ satisfies a structural condition similar to~\eqref{eq:FPstr}:
\begin{equation} \label{eq:disFPstr}
\tcm{\|U^{n}\|_{L^{s}(\Omega)}^{s-1+\gamma}\leq C_s \int_{\Omega} \left( U^{n}(x) \right)^{s-1} \mathcal{N}[U^{n}](x)dx, }
\end{equation}
where $\gamma\in(0, \infty)$ 
and some constant $\tcm{C_s}>0$ that is independent of $n$ and $h$. 

Under these hypotheses, our aim is to derive results for $\|U^n\|_{ L^{p}(\Omega)}$ to show that this semidiscrete solution has the same long-time power-law decay as the continuous solution in Lemma~\ref{lem:non-decay2}.
We begin with the following technical lemma.

\begin{lemma}\label{lem:keyineq}
\tcm{Let $s$ be as in \eqref{eq:disFPstr}. Assume that $U^{n}\in L^{s}(\Omega)$ for all $n\geq 0$.}
Set $V^{n}=\|U^{n}\|_{\tcm{L^{s}}(\Omega)}$. Let $\mathcal{D}_{h}^{\alpha}(\cdot)$ be a $\mathcal{CM}$-preserving scheme.
Then 
\begin{equation} \label{eq:keyineq}
\begin{split}
\tcm{\left( V^{n} \right)^{s-1} \mathcal{D}_{h}^{\alpha} \left( V^{n} \right) \leq \int_{\Omega} \left( U^{n}(x) \right)^{s-1} } \mathcal{D}_{h}^{\alpha} \left( U^{n} \right)(x)\,dx\ \text{ for }n= 1,2,\dots
\end{split}
\end{equation}
\end{lemma}
\begin{proof}%[proof of Lemma \ref{lem:keyineq}]
For $n\ge 1$ it follows from \eqref{eq:CMsch} that
\begin{equation*} %\label{eq:keyineq}
\begin{split}
\int_{\Omega} &\left( U^{n}(x) \right)^{s-1} \cdot \mathcal{D}_{h}^{\alpha} \left( U^{n} \right)(x)dx
= \int_{\Omega} \left( U^{n}(x) \right)^{s-1} \cdot \frac{1}{h^{\alpha}}
	\left(\sum\limits_{k=0}^{n-1}\omega_{j}U^{n-j}+\delta_{n}U^{0}\right) (x)dx\\
 &=\frac{1}{h^{\alpha}}\left(\omega_{0}\|U^{n}\|^{s}_{L^{s}(\Omega)}+\sum\limits_{k=1}^{n-1}\omega_{j} \int_{\Omega} \left( U^{n}(x) \right)^{s-1} U^{n-j}(x)dx +\delta_{n} \int_{\Omega} \left( U^{n}(x) \right)^{s-1}U^{0} (x)dx \right)\\
 &\geq \frac{1}{h^{\alpha}}\left(\omega_{0} \|U^{n}\|^{s}_{L^{s}(\Omega)}+\sum\limits_{k=1}^{n-1}\omega_{j}  \left( \|U^{n-j}\|_{L^{s}(\Omega)} \|(U^{n})^{s-1}\|_{L^{q}(\Omega)} \right) +\delta_{n} \left( \|U^{0}\|_{L^{s}(\Omega)} \|(U^{n})^{s-1}\|_{L^{q}(\Omega)} \right) \right),\\
\end{split}
\end{equation*}
using  $\delta_{n}<0$ for $n\ge 1$, $\omega_{k}<0$ for $k\geq 1$, and H\"older's inequality with $q$ defined by 
$\frac{1}{s}+\frac{1}{q}=1$.
But $\|(U^{n})^{s-1}\|_{L^{q}(\Omega)}=\left( \int_{\Omega} \left| U^{n}(x) \right|^{q(s-1)}dx\right)^{1/q}= \left( \int_{\Omega} \left| U^{n}(x) \right|^{\frac{s}{s-1} (s-1)}dx\right)^{\frac{s-1}{s} } =\|(U^{n})\|^{s-1}_{L^{s}(\Omega)}$,
so we get 
\begin{equation*} 
\begin{split}
\int_{\Omega} \left( U^{n}(x) \right)^{s-1} \cdot \mathcal{D}_{h}^{\alpha} \left( U^{n} \right)(x)dx
 &\geq \|(U^{n})\|^{s-1}_{L^{s}(\Omega)}  \cdot \frac{1}{h^{\alpha}}\left(\omega_{0} \|U^{n}\|_{L^{s}(\Omega)}+\sum\limits_{k=1}^{n-1}\omega_{j}   \|U^{n-j}\|_{L^{s}(\Omega)}  +\delta_{n} \|U^{0}\|_{L^{s}(\Omega)}  \right)\\
 &=\left( V^{n} \right)^{s-1}\mathcal{D}_{h}^{\alpha} \left( V^{n} \right).
\end{split}
\end{equation*}
\end{proof}

The discrete energy-type inequality of this lemma is often a key step in the analysis of F-PDEs. An alternative discrete energy inequality, where $\tcm{s=2}$, was proved for the Gr\"{u}nwald-Letnikov  and L1 discretisations  applied to initial-value problems  $\mathcal{D}_{t}^{\alpha}y(t)= f(t,y(t))$ in~\cite[Lemma 3.2]{wang2018long}. 

The proof of Lemma~\ref{lem:keyineq} depends only on the helpful signs of the coefficients in $\mathcal{CM}$-preserving schemes; the bounds of~\eqref{eq:coeffdecay} are not needed. This important observation shows that this lemma remains valid if we replace the $\mathcal{CM}$-preserving scheme by the L1 discretization that was defined in \eqref{eq:L1a} for nonuniform meshes. 
\begin{lemma}\label{lem:keyineq2}
Assume the hypotheses of Lemma~\ref{lem:keyineq}. Then for the L1 scheme defined in \eqref{eq:L1a} on an arbitrary mesh, one has
\begin{equation} \label{eq:keyineq2}
\begin{split}
\tcm{\left( V^{n} \right)^{s-1} \mathcal{D}_{L1}^{\alpha} \left( V^{n} \right) \leq \int_{\Omega} \left( U^{n}(x) \right)^{s-1} }  \mathcal{D}_{L1}^{\alpha} \left( U^{n} \right)(x)\,dx\ \text{ for }n= 1,2,\dots
\end{split}
\end{equation}
\end{lemma}

\begin{proof}
Similarly to the proof of Lemma \ref{lem:keyineq},  for $n\ge 1$ it follows from \eqref{eq:L1a} that
\begin{equation*} %\label{eq:keyineq}
\begin{split}
\int_{\Omega} &\left( U^{n}(x) \right)^{s-1} \cdot \mathcal{D}_{L1}^{\alpha} \left( U^{n} \right)(x)\,dx
= \int_{\Omega} \left( U^{n}(x) \right)^{s-1} \frac1{\Gamma(2-\alpha)} \left[ d_{n,1} U_n - d_{n,n}U_0 
	+ \sum^{n-1}_{k=1}U_{n-k} \left(d_{n,k+1}-d_{n,k}\right)  \right]  dx,\\
\end{split}
\end{equation*}
where the scheme coefficients have the helpful property $0<d_{n,k+1}<d_{n,k}$ for all $k$ and $n$.  One can use this property to complete the argument by imitating the proof of Lemma~\ref{lem:keyineq}; we omit the details.
\end{proof}

Our main result in this section is the following theorem, which is a discrete analogue of Lemma~\ref{lem:non-decay2}.

\begin{theorem}\label{thm:disFPDE}
Let $\mathcal{D}_{h}^{\alpha}$ be a $\mathcal{CM}$-preserving scheme and let $U^{n}\ge 0$ be the solution of \eqref{eq:disFPDEs} for $n=0,1,\dots$ with the structural condition \eqref{eq:disFPstr}. Then 
\begin{equation} \label{eq:FPDEesti}
\begin{split}
    \mathcal{D}_{h}^{\alpha} \left( \|U^{n}\|_{\tcm{L^{s}(\Omega)}} \right) \leq - \frac{\|U^{n}\|_{\tcm{L^{s}(\Omega)}}^{\gamma}}{\tcm{C_s}}\ \text{ for }n= 1,2,\dots
\end{split}
\end{equation}
Furthermore, there exists a positive constant $h_0$ such that for any $h$ satisfying $0<h\le h_0$, one has
\begin{equation} \label{eq:FPdecay}
\begin{split}
 \|U^{n}\|_{\tcm{L^{s}(\Omega)}} \leq \frac{C_d^{*}}{1+t_{n}^{\alpha/\gamma}} \ \text{ for }n= 1,2,\dots,
\end{split}
\end{equation}
where the constant $\tcm{C_s^*=C_s^*(C_s,\gamma, \alpha, U^0)}>0$. 

%\tcm{
If the $\mathcal{CM}$-preserving scheme in \eqref{eq:disFPDEs} is replaced by the L1 scheme \eqref{eq:L1a} on an arbitrary mesh, then one again obtains \eqref{eq:FPDEesti}.  If in addition this mesh satisfies Assumption~\ref{ass:meshK},  then one also obtains~\eqref{eq:FPdecay}.
%}
\end{theorem}
\begin{proof}
Let $n\ge 1$. Multiplying both sides of the PDE in~\eqref{eq:disFPDEs} by $\left( U^{n}(x) \right)^{s-1} $ and integrating with respect to $x$, we obtain 
\begin{equation*} 
\begin{split}
 \int_{\Omega} \left( U^{n}(x) \right)^{s-1} \mathcal{D}_{h}^{\alpha} \left( U^{n} \right)(x)dx=-\int_{\Omega} \left( U^{n}(x) \right)^{s-1} \mathcal{N}[U^{n}](x)dx.
\end{split}
\end{equation*}
Now an appeal to Lemma~\ref{lem:keyineq} and the structural assumption \eqref{eq:disFPstr} yield
\begin{equation*} 
\begin{split}
\|U^{n}\|_{L^{s}(\Omega)}^{s-1} \mathcal{D}_{h}^{\alpha} \left( \|U^{n}\|_{L^{s}(\Omega)} \right) \leq -\frac{\|U^{n}\|_{L^{s}(\Omega)}^{s-1+\gamma}}{C_d}.
\end{split}
\end{equation*}
The inequality \eqref{eq:FPDEesti} follows immediately if $\|U^{n}\|_{L^{s}(\Omega)}\neq 0$. If  $\|U^{n}\|_{L^{s}(\Omega)}= 0$, then 
\begin{equation*} 
\begin{split}
\mathcal{D}_{h}^{\alpha} \left( U^{n}\right) 
= \frac{1}{h^{\alpha}}\left(\sum_{k=0}^{n-1}\omega_{k}U^{n-k}+\delta_{n}U^{0}\right)\leq 0
\end{split}
\end{equation*}
because of the hypothesis $U^{k}\geq 0$ for all $k$ and the scheme properties $\delta_{n}<0$ and $\omega_{k}<0$ for $k \geq 1$. Thus we have again obtained~\eqref{eq:FPDEesti}.

Now that \eqref{eq:FPDEesti} has been proved, the inequality \eqref{eq:FPdecay} follows immediately from Lemma~\ref{lem:Liwang21} and the supersolution analysis in Theorem~\ref{thm:main}. 

%\tcm{
The proof of  \eqref{eq:FPDEesti} for the L1 scheme \eqref{eq:L1a} on an arbitrary mesh  follows from Lemma \ref{lem:keyineq2}  in a similar manner.  \tcm{As we pointed out in Section~\ref{sec:L1}, Lemma~\ref{lem:Liwang21} is easily extended to the L1 scheme on arbitrary meshes since its proof depends only on the signs of the coefficients in the scheme.
Hence} if the mesh satisfies Assumption~\ref{ass:meshK}, then the supersolution analysis of Section~\ref{sec:discsuper} yields  \eqref{eq:FPdecay}.
%}
\end{proof}

Many concrete examples that satisfy the structural assumptions \eqref{eq:FPstr} or \eqref{eq:disFPstr} can be found in \cite{vergara2015optimal, dipierro2019decay, affili2019decay}. 
We will use these specific examples to support our theoretical results by numerical experiments.

\section{Applications and numerical experiments}\label{sec:NumExp}

\subsection{Nonlinear  F-ODEs}\label{sec:NumFODE}
Recall the scalar model F-ODE \eqref{FODE}: $\mathcal{D}_{t}^{\alpha}y(t)=-\lambda y(t)^{\gamma}$ for $t>0$,
with initial value $y(0)=y_0>0$. 
In this section we verify numerically the optimal numerical decay rates that are predicted by Theorem~\ref{thm:main} for $\mathcal{CM}$-preserving methods applied to this problem. Various values of the parameters and initial values will be tested. 

\begin{figure}[!ht]
\begin{center}
$\begin{array}{cc}
 \includegraphics[scale=0.41]{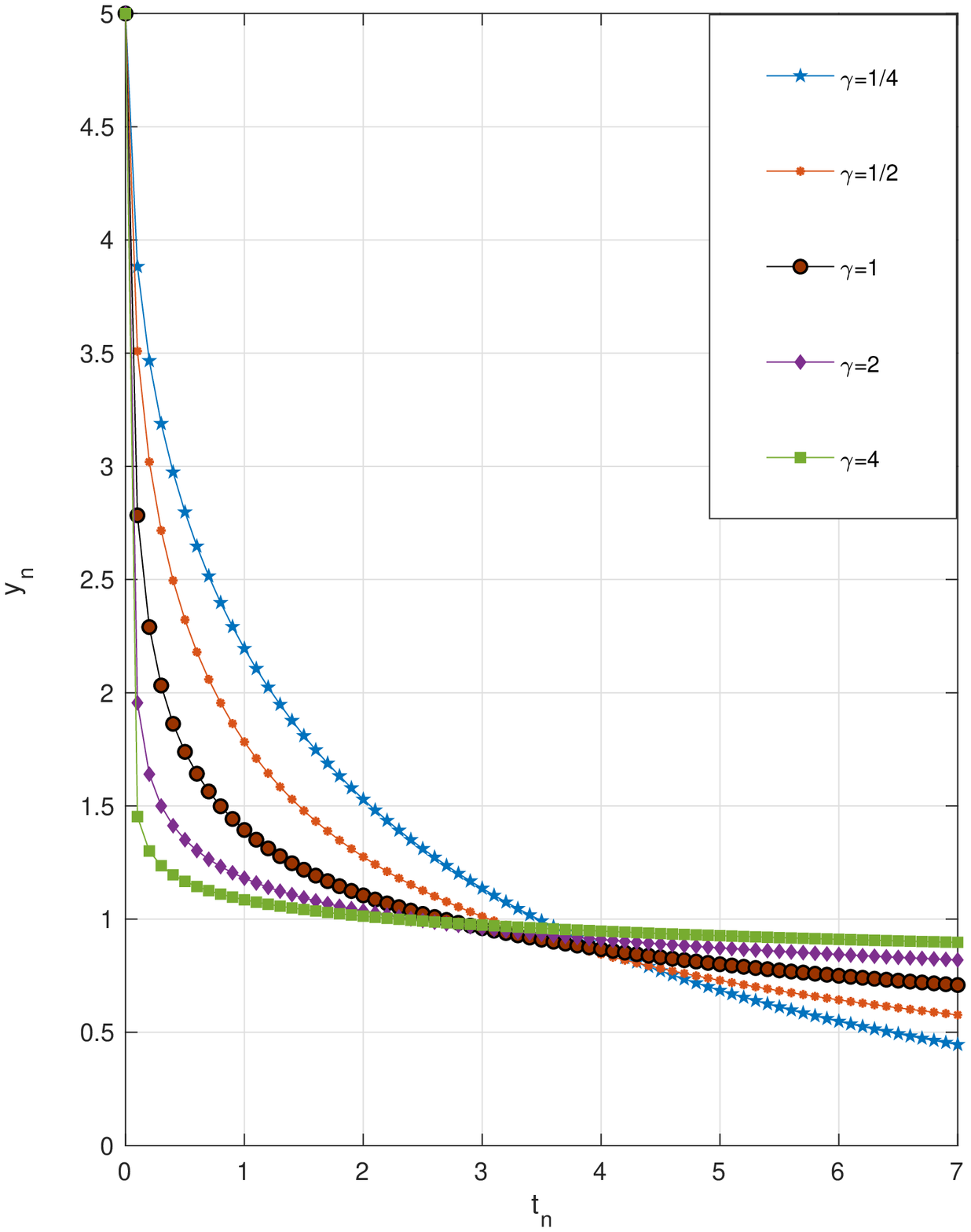}\quad
  \includegraphics[scale=0.45]{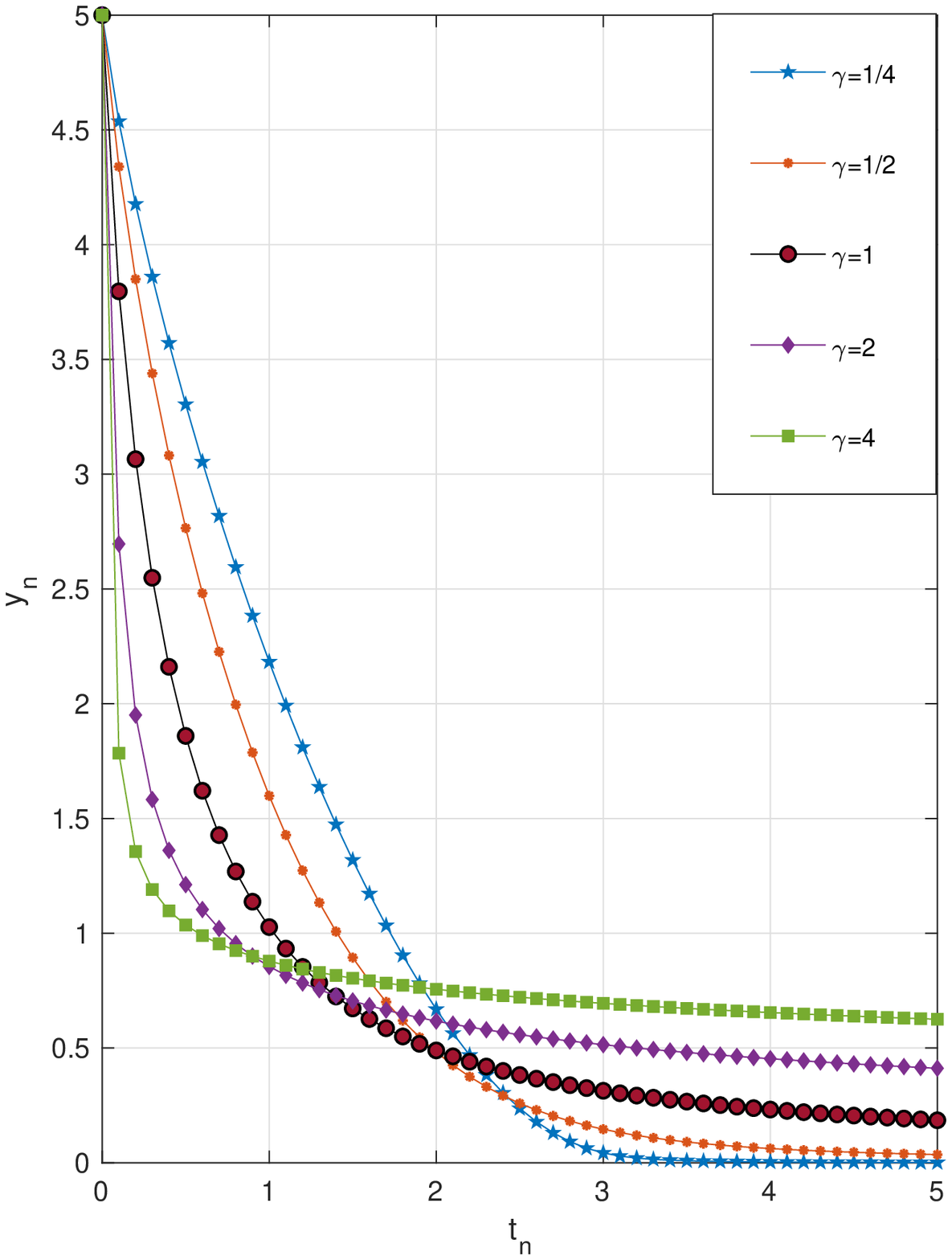}
 \end{array}$
\end{center}
\caption{Numerical solutions for $\alpha=0.4$ (left) and $\alpha=0.8$ (right) with parameter $\gamma=1/4, 1/2, 1, 2, 4$.
}
\label{ex11}
\end{figure}

Figure~\ref{ex11} plots the numerical solutions computed by the Gr\"unwald-Letnikov scheme for various values of~$\alpha$ and~$\gamma$, where $\lambda=2, h=0.1$ and $y_0=5$. 
We see that all numerical solutions maintain monotonicity and positivity,  
as expected.  \tcm{Here we graph the numerical solutions only for $0<t_{n}<10$, in order to exhibit more clearly the differing decay rates of the computed solutions for various values of~$\gamma$. These decay rates are preserved on longer intervals such as $0 < t_n < 50$.}
To test the numerical decay rate quantitatively, we introduce the index function 
\begin{equation} \label{eq:indexp}
\begin{split}
q_{\alpha,\lambda}(t_{n})=-\frac{\ln(|y_{n}|/|y_{n-k}|)}{\ln(t_{n}/t_{n-k})}\ \text{ for }t_{n-k}>1\text{ and } k\in\mathbb{N}^{+}.
\end{split}
\end{equation}
This quantity  is a numerical observation of the ratio $\alpha/\lambda$ when $\|y_{n}\|=O(t_{n}^{-\alpha/\lambda})$ as $t_n\to \infty$.
A similar index was used in~\cite{wang2018long, wang2021mittag}.
Table \ref{tableexam1} shows that the numerical observations $q_{\alpha,\lambda}(t_n)$ agree with the theoretical prediction $q_{\alpha,\lambda}^{*}(t_n):=\alpha/\gamma$ of Theorem~\ref{thm:main}; thus, the numerical solutions have a polynomial decay rate, which is very different from the exponential decay rate of solutions to integer-order ODEs. 

These results show that the  Gr\"unwald-Letnikov scheme captures accurately the decay rate of nonlinear F-ODEs, which is exactly consistent with our theoretical prediction. Other $\mathcal{CM}$-preserving schemes, such as the L1 scheme, yield results that are very similar to those of the Gr\"unwald-Letnikov scheme, so we do not present them here.

\begin{table}
\caption{Observed $q_{\alpha, \lambda}(t_{n})$  for $t_{n}= 10,20,30,40,50$ with $\alpha=0.4$ and $\alpha=0.8$ (in parentheses)}
%\tcm{
\begin{center}
\begin{tabular}{lll lll ll }
\hline $t_n$  & $\gamma=1/4$ & $\gamma=1/2$  & $\gamma=1$ & $\gamma=2$ & $\gamma=4$\\
\hline
$10   $ & 1.5949 (3.7105) & 0.7515 (1.8619)  & 0.3727 (0.9154) & 0.1898 (0.4260) & 0.0973 (0.2044)\\
$20   $ & 1.7528 (3.4124) & 0.7901 (1.7053)  &0.3789 (0.8588)  & 0.1900 (0.4166)  & 0.0967 (0.2022)\\
$30  $ & 1.7291 (3.3344) & 0.8002 (1.6656)  &0.3820 (0.8407)  & 0.1904 (0.4130)  & 0.0966 (0.2014)\\
$40  $& 1.7042  (3.2984) & 0.8040 (1.6476)  & 0.3840 (0.8315) & 0.1907 (0.4111)  & 0.0966 (0.2011)\\
$50 $& 1.6861 (3.2776) & 0.8057 (1.6372) & 0.3854 (0.8260) & 0.1910 (0.4098) & 0.0966 (0.2008)\\
\hline
$q_{\alpha, \lambda}^{*}$ &  1.6 (3.2) &  0.8 (1.6)  &  0.4 (0.8)  &  0.2 (0.4)  &  0.1 (0.2) \\
\hline
\end{tabular}
\end{center}
%}
\label{tableexam1}
\end{table}

\subsection{Nonlinear  F-PDEs}\label{sec:NumFPDE}
%\tcm{
The estimates of the decay rate of the true solution in Lemma~\ref{lem:non-decay2} and its discrete analogue in  Theorem~\ref{thm:disFPDE} are independent of the dimension $d$ of the spatial region~$\Omega$. In our numerical examples we shall take $d=2$. Thus, consider
%}
F-PDEs of the form~\eqref{eq:FPDEs} with $\Omega\subseteq \mathbb{R}^2$ a bounded convex domain, 
where we discretize in space by a conforming piecewise linear finite element method and in time by a $\mathcal{CM}$-preserving scheme. 

First, we describe briefly  the iterative  numerical algorithm for nonlinear fully discrete schemes. Consider the model F-PDE
\begin{equation}\label{modelFPDE}
\mathcal{D}_{t}^{\alpha}u(x, t)=\nabla\cdot \left( a(u, \nabla u)\nabla u\right)\ \text{ for }t>0,\ x\in \Omega,
\end{equation}
where $a(u, \nabla u)$ is a known nonlinear function; this F-PDE is subject to a Dirichlet boundary condition $u(x, t)=0$ for $x\in \partial \Omega$ and has initial value $u(x, 0)=u_0$, . 
 
The time discretisation of~\eqref{modelFPDE} yields $\mathcal{D}_{h}^{\alpha}U^{n}(x)=\nabla\cdot \left( a(U^{n}, \nabla U^{n})\nabla U^{n}\right)$ for $n\geq 1$, where $U^{n}(x)\approx u(x, t_{n})$ at $t_n=nh$ with time-step $h>0$.

Set $\beta:= h^{\alpha}/\omega_0$. For the particular value $n=1$,  the discretisation is $\mathcal{D}_{h}^{\alpha}U^{1}(x)=(1/\beta) \left(U^{1}(x)-U^{0}(x)\right)$.  This gives $U^1(x)=U^0(x)+\beta \nabla\cdot ( a(U^{1}(x), \nabla U^{n}(x))\nabla U^{1}(x))$.
 
 For $n\geq 2$,  one has $\mathcal{D}_{h}^{\alpha}U^{n}(x)=\frac{1}{h^{\alpha}}(\sum_{k=0}^{n-1}\omega_{k}U^{n-k}+\delta_{n}U^{0})$, 
 where $\delta_{n}=-\sum_{k=0}^{n-1}\omega_{k}$; see equation \eqref{eq:CMsch}. This yields  
 $U^n(x)=\beta \nabla\cdot \left( a(U^{n}(x), \nabla U^{n}(x))\nabla U^{n}(x)\right)-b_n,$
 where $b_n:=\frac{1}{\omega_0} (\sum_{k=1}^{n-1}\omega_{k}U^{n-k}+\delta_{n}U^{0})$. Then for each $n\geq 1$, we need to solve numerically the nonlinear elliptic-type equation 
\begin{equation}\label{eq:ellipticmod}
\begin{split}
 U^n(x)=\beta \nabla\cdot \left( a(U^n(x), \nabla U^n(x))\nabla U^n(x)\right)-b_n,
\end{split}
\end{equation}
where we set $b_1:=-U^0$.

Set $W:=U^n(x)$ for brevity in what follows.  The equation~\eqref{eq:ellipticmod} is solved using a piecewise linear finite element method in the open source computing software system FEALPy \cite{fealpy}.
To do this, multiply both sides of~\eqref{eq:ellipticmod} by the test function $V\in H_{0}^{1}(\Omega)$, and obtain a weak form of the equation by integration by parts:
\begin{equation}\label{eq:weakformu}
\begin{split}
 (W , V)+\beta \left( a(W, \nabla W)\nabla W, \nabla V\right)=-(b_n, V),
\end{split}
\end{equation}
where $(f,g):=\int_{\Omega}fg dx$. 
Let $\tilde W$ be some approximation to $W$ from the finite element space. Set $\delta W=W-\tilde W$. Then \eqref{eq:weakformu} is equivalent to 
\begin{equation}\label{eq:weakformuB}
\begin{split}
 (\tilde W+\delta W, V)+\beta \left( a(\tilde W+\delta W, \nabla (\tilde W+\delta W))\nabla (\tilde W+\delta W), \nabla V\right)=-(b_n, V).
\end{split}
\end{equation}
Take the Taylor expansion of $a(\tilde W+\delta W, \nabla (\tilde W+\delta W))$ about $(\tilde W, \nabla\tilde W)$:
\begin{equation}\label{eq:Tay}
\begin{split}
a(\tilde W+\delta W, \nabla (\tilde W+\delta W))=a(\tilde W, \nabla \tilde W)+J_{W} \delta W+J_{\nabla W} \cdot \nabla \delta W+O(\delta W^2+|\nabla \delta W|^2),
\end{split}
\end{equation}
where $J_{W}:= \partial a(W, \nabla W)/\partial W$ and $J_{\nabla W}:= \partial a(W, \nabla W)/\partial \nabla W$, with both of these functions evaluated at the initial approximation $(\tilde W, \nabla \tilde W)$.
Substituting \eqref{eq:Tay} into \eqref{eq:weakformu} and removing the higher-order terms, we get  
\[
%\begin{equation}\label{eq:weakformu2}
 (\tilde W+\delta W, V)+\beta \left( a(\tilde W, \nabla \tilde W)\nabla \tilde W+ a(\tilde W, \nabla \tilde W)\nabla \delta W+ J_{W} \nabla \tilde W\delta W+J_{\nabla W} \cdot \nabla \tilde W\nabla \delta W, \nabla V  \right)=-(b_n, V).
%\end{equation}
\]
This is a discrete linear system for $\delta W$, which we solve approximately in our piecewise linear finite element space; then we update our finite element approximation of $W$ by $\tilde W\mapsto \tilde W + \delta W$.

Using Picard iteration, we repeat this updating process until we obtain a numerical approximation to $W$ with a predetermined precision 
and use this as the computed solution for equation \eqref{eq:ellipticmod}.

In this model, $x=(x_1, x_2)^{T}$ and $\nabla u=(\partial_{x_1} u, \partial_{x_2} u)^{T}$.
In the example of Section~\ref{sec:NumpLap}, we have $a(u, \nabla u)=|\nabla u|^{p-2}=(\nabla u, \nabla u)^{(p-2)/2} $,  
$J_{U}= \partial a(\tilde U, \nabla \tilde U)/\partial U =0$
and 
\[
J_{\nabla U}=\left. \frac{\partial}{\partial\nabla U} (\nabla U, \nabla U)^{(p-2)/2}\right|_{\tilde U}
	=\left. (p-2) \left[|\nabla U|^{p-4} \nabla U\right] \right|_{\tilde U}.
\]
Similarly,  in the example of Section~\ref{sec:Numfpme} we have $a(u, \nabla u)= 1/\sqrt{1+|\nabla u|^2}$, $J_{U}=0$
and 
\[
J_{\nabla U}=\left. \frac{\partial}{\partial\nabla U} \left( \frac{1} {\sqrt{1+(\nabla U, \nabla U)}}\right) \right|_{\tilde U}
=\left. - \left[(1+|\nabla U|^2)^{-\frac{3}{2}} \nabla U\right] \right|_{\tilde U}.
\]

\subsubsection{Fractional $p$-Laplace equation}\label{sec:NumpLap}
Let $p\in(1, \infty)$.
The time fractional $p$-Laplace problem is 
\begin{equation}\label{eq:FAC}
\begin{split}
 \left\{
\begin{aligned}
\mathcal{D}_{t}^{\alpha}u(x, t)&=\Delta_{p}(u)\ \text{ for }t>0 \text{ and } x\in \Omega,\\
u(x, t)&=0\ \text{ for }t>0 \text{ and } x\in \partial \Omega,\\
u(x, 0)&=u_0(x) \ \text{ for }x\in \Omega,
\end{aligned}
\right.
\end{split}
\end{equation}
where $\Delta_{p}(u)=\nabla\cdot (|\nabla u|^{p-2} \nabla u)$ and $\mathcal{N}[u]=-\Delta_{p}(u)$. 
It is shown in \cite{vergara2015optimal, dipierro2019decay} that for any $s>1$ one has
$\|u\|^{s-2+p}_{L^{s}(\Omega)}(t)\leq C \int_{\Omega} \nabla u^{s-1}\cdot \left( |\nabla u|^{p-2} \nabla u\right)dx$, where the constant $C=C( \alpha, s, u_0, \Omega,p)>0$, i.e., the structural assumption \eqref{eq:FPstr} holds true for $\gamma=p-1$.
Applying the previous general result given in Lemma \ref{lem:non-decay2}, we can get the following estimates. 

\begin{lemma} \cite[Theorem 1.2]{dipierro2019decay} \label{lem:Fplap}
Suppose that $u$ is a solution of \eqref{eq:FAC}. Then for \tcm{each $s\in(1, \infty)$, 
there exists a positive constant $C$ such that}
$ \|u\|_{L^{s}(\Omega)}(t) \leq \frac{C}{1+t^{\alpha/(p-1)}} \text{ for all } t>0.$
\end{lemma}
%\tcm{
This result corresponds to $\gamma=p-1$ in Lemma~\ref{lem:non-decay2}.
%}
\tcm{Analogously to the continuous case, we assume that for some $s\in (1, \infty)$,  $U^n\in L^s(\Omega)$ satisfies the structural condition \eqref{eq:disFPstr} with $\gamma=p-1$.}
Then the numerical solutions satisfy
\begin{equation*} 
\begin{split}
 \|U^{n}\|_{L^{s}(\Omega)} \leq \frac{C}{1+t_{n}^{\alpha/(p-1)}} \text{ for all } t_{n}>0.
\end{split}
\end{equation*}
where $U^{n}$ is the numerical approximation of $u(x, t)$ at $t_{n}$ given in Theorem \ref{thm:disFPDE}.

%\tcm{
In the numerical experiments we take $\Omega=[0, 1]^2$, the initial value $u_{0}(x_1,x_2)=\sin(\pi x_1)\sin(\pi x_2)$, step size $h=1/150$ and $s=2$ to compute $\|U^{n}\|_{L^{2}(\Omega)}$. Similarly to~\eqref{eq:indexp}, we define 
\begin{equation} \label{eq:indexq}
\begin{split}
q_{\alpha,p}(t_{n})=-\frac{\ln(\|U^{n}\|_{L^{2}(\Omega)}/\|U_{n-k}\|_{L^{2}(\Omega)})}{\ln(t_{n}/t_{n-k})}\ \text{ for }t_{n-k}>1\text{ and } k\in\mathbb{N}^{+}.
\end{split}
\end{equation}
Here the index $q_{\alpha, p}$ is the numerical observation decay rate; \tcm{the decay rate $q_{\alpha, p}^{*}$ for both the exact solution  and numerical solution satisfies $q_{\alpha, p}^{*}\geq \alpha/(p-1)$, according to Lemma~\ref{lem:Fplap} and Theorem~\ref{thm:disFPDE} respectively.}
In order to reduce the observation error, 
\tcm{for each $t_n$ we take $k=2,4,6,8,10$ to compute 5 different $q_{\alpha,p}(t_{n})$ from \eqref{eq:indexq}, and then compute the mean of these 5 values as the observed value at~$t_n$.
Tables~\ref{tableexam21} and~\ref{tableexam22} show that the observed values $q_{\alpha, p}$ of the numerical decay rate  are consistent with our theoretical prediction $q_{\alpha, p}^{*}\geq \alpha/(p-1)$.}

\begin{table}
\caption{Observed $q_{\alpha, p}$ with $p=3$ and various values of $\alpha$ for $t_{n}=10,15,20$} 
%\tcm{
\begin{center}
\begin{tabular}{lll lll ll }
\hline $t_n$  & $\alpha=0.1$ & $\alpha=0.2$  & $\alpha=0.4$ & $\alpha=0.7$ & $\alpha=0.9$\\
\hline
$10   $ & 0.0509 & 0.1032  & 0.2053 & 0.3573 & 0.4317\\
$15   $ &0.0558 & 0.0921  &0.2002  & 0.3437  & 0.4704\\
$20  $ & 0.0526 & 0.0945  &0.1988  & 0.4060  & 0.4775\\
\hline
$q_{\alpha, p}^{*}$ &  0.05 &  0.1  &  0.2  &  0.35  &  0.45 \\
\hline
\end{tabular}
\end{center}
%}
\label{tableexam21}
\end{table}

\begin{table}
\caption{Observed $q_{\alpha, p}$ with $\alpha=0.9$ and various values of $p$ for $t_{n}=10,15,20$} 
%\tcm{
\begin{center}
\begin{tabular}{lll lll ll }
\hline $t_n$  & $p=1.5$ & $p=2$  & $p=3$ & $p=4$ & $p=5$\\
\hline
$10   $ & 1.8083 & 0.9102  & 0.4317 & 0.3116 & 0.2316\\
$15   $ & 1.8055& 0.9070  &0.4704  & 0.3101  & 0.2310\\
$20  $ & 1.8041 & 0.9054  &0.4775 & 0.3092  & 0.2305\\
\hline
$q_{\alpha, p}^{*}$ &  1.8 &  0.9  &  0.45  &  0.30  &  0.215 \\
\hline
\end{tabular}
\end{center}
%}
\label{tableexam22}
\end{table}

\subsubsection{Fractional mean curvature equation}\label{sec:Numfpme}
Consider the nonlinear time fractional mean curvature equation
\begin{equation}\label{eq:FMC}
\begin{split}
 \left\{
\begin{aligned}
\mathcal{D}_{t}^{\alpha}u(x, t)&=\nabla\cdot \left( \frac{\nabla u} {\sqrt{1+|\nabla u|^2}}\right)\ \text{ for }t>0 \text{ and } x\in \Omega,\\
u(x, t)&=0\ \text{ for }t>0 \text{ and } x\in \partial \Omega,\\
u(x, 0)&=u_0(x)\ \text{ for } x\in \Omega.
\end{aligned}
\right.
\end{split}
\end{equation}
Here the right-hand side of the PDE in~\eqref{eq:FMC} is the mean curvature of the hypersurface described by the graph of the function $u$. 
It is shown in \cite{dipierro2019decay} that for any $s>1$,  
$\|u\|^{s}_{L^{s}(\Omega)}(t)\leq C \int_{\Omega} \frac{\nabla u \cdot \nabla u^{s-1}} {\sqrt{1+|\nabla u|^2}} dx$, where the constant $C=C( \alpha, s, u_0, \Omega)>0$.  
That is, the structural assumption \eqref{eq:FPstr} holds true with $\gamma=1$, and we can invoke Lemma~\ref{lem:non-decay2} to derive a long-time decay estimate.

\begin{lemma}\cite[Theorem 1.5]{dipierro2019decay} \label{lem:Fmeancurv}
Let $u$ be a solution 
%\footnote{M: ``a" solution or ``the" solution, i.e., is is known whether the solution is unique?}
of \eqref{eq:FMC} and assume that
$\sup_{x\in\Omega, t>0} |\nabla u(x, t)| <\infty$.
Then for any $s\in(1, \infty)$, $ \|u\|_{L^{s}(\Omega)}(t) \leq \frac{C}{1+t^{\alpha}}$, $t>0$.
\end{lemma}

Like the continuous case, \tcm{if we assume that the time discrete scheme satisfies the structural assumption \eqref{eq:disFPstr} for some $s\in (1,\infty)$ with $\gamma=1$, it then follows} from Theorem~\ref{thm:disFPDE} that
\begin{equation*} 
\begin{split}
 \|U^{n}\|_{L^{s}(\Omega)} \leq \frac{C}{1+t_{n}^{\alpha}}\ \text{ for all } t_{n}>0.
\end{split}
\end{equation*}
\tcm{Hence $q_{\alpha}^{*}\geq \alpha$, where $q_{\alpha}^{*}$ is the predicted decay rate of the numerical solution.}

\begin{table}
\caption{Observed $q_{\alpha}$ for various values of $\alpha$ for $t_{n}=10,15,20$} 
\begin{center}
\begin{tabular}{ll lll ll }
\hline $t_n$  & $\alpha=0.2$  & $\alpha=0.4$ & $\alpha=0.7$ & $\alpha=0.9$\\
\hline
$10   $  & 0.3495  & 0.4664 & 0.7479 & 0.9655\\
$15   $ & 0.3135  &0.4445  & 0.7306  & 0.9421\\
$20  $  & 0.2948  &0.4338  & 0.7226  & 0.9311\\
\hline
$q_{\alpha}^{*}$  &  0.2  &  0.4  &  0.7  &  0.9 \\
\hline
\end{tabular}
\end{center}
\label{tableexam31}
\end{table}

In the numerical experiments we take $\Omega=[0, 1]^2$, the initial value $u_{0}(x_1,x_2)=10 \sin(\pi x_1)\sin(\pi x_2)$, step size $h=1/150$ and $s=2$ to compute $\|U^{n}\|_{L^{2}(\Omega)}$.
Table~\ref{tableexam31} shows that the observed values of the numerical decay rates \tcm{(which are computed similarly to Section~\ref{sec:NumpLap}) agree with the above prediction that $q_{\alpha}^{*}\geq \alpha$.} \\ 
 
\tcm{
\emph{Conclusion.}
We emphasise that our numerical results for the nonlinear F-PDEs of Section~\ref{sec:NumExp} exhibit polynomial decay rates that are  typical  of long-time solutions to time-fractional problems. 
In this paper, we have considered the decay rates of exact and computed solutions for this class of problems. The more difficult question of convergence of computed solutions to exact solutions will be examined in a future paper. }

\section*{Acknowledgement}
The authors are grateful to Dr.~Chunyu Chen and Professor Huayi Wei \tcm{of Xiangtan University} for much help in the implementation of numerical experiments.

\bibliographystyle{alpha}
\bibliography{NonFODE}

\end{document}